\documentclass[12pt]{article}
\usepackage{amsmath}
\usepackage{amsfonts}
\usepackage{graphicx,psfrag,epsf}
\usepackage{enumerate}
\usepackage{url} % not crucial - just used below for the URL 
\usepackage{type1ec}
\usepackage{latexsym}
\usepackage{epsfig, ecltree, epic, eepic}
\usepackage{enumerate,amsmath,amssymb,dsfont,pifont,amsthm}
\usepackage[dvips]{color}
\usepackage{graphicx, float}
\usepackage[arrow, matrix, curve]{xy}

\usepackage{pstricks}
\usepackage{pst-grad} % For gradients
\usepackage{pst-plot} % For axes

\newcommand{\blind}{0}

\theoremstyle{plain}
\newtheorem{theorem}{Theorem}[section]
\newtheorem{corollary}[theorem]{Corollary}
\newtheorem{lemma}[theorem]{Lemma}

\theoremstyle{definition}

\theoremstyle{remark}
\newtheorem{remark}[theorem]{Remark}
\newtheorem{example}[theorem]{Example}

\newcommand{\bbr}{\mathbb{R}}

\newcommand{\bbn}{\mathbb{N}}
\newcommand{\bbz}{\mathbb{Z}}
\newcommand{\cb}{\mathcal{B}}
\newcommand{\ca}{\mathcal{A}}
\newcommand{\cf}{\mathcal{F}}

\newcommand{\abs}[1]{\left| #1 \right|}
\newcommand{\norm}[1]{\left\| #1 \right\|}
\newcommand{\Cov}{\operatorname{Cov}}
\newcommand{\Var}{\operatorname{Var}}
\newcommand{\ord}{\operatorname{ord}}		

\begin{document}

\def\spacingset#1{\renewcommand{\baselinestretch}%
{#1}\small\normalsize} \spacingset{1}

%%%%%%%%%%%%%%%%%%%%%%%%%%%%%%%%%%%%%%%%%%%%%%%%%%%%%%%%%%%%%%%%%%%%%%%%%%%%%%

\if0\blind
{
  \title{\bf Testing for Structural Breaks via Ordinal Pattern Dependence}
  \author{Alexander Schnurr\thanks{
    The authors gratefully acknowledge financial support of the DFG (German science Foundation) SFB 823: Statistical modeling of nonlinear dynamic processes (projects C3 and C5).}\hspace{.2cm}\\
    Fakult\"at f\"ur Mathematik, Technische Universit\"at Dortmund
              %D-44227 Dortmund, Germany,
              %\texttt{alexander.schnurr@math.tu-dortmund.de}\\
    and \\
    Herold Dehling$^*$ \hspace{.2cm}\\
    Fakult\"at f\"ur Mathematik, Ruhr-Universit\"at Bochum
    }
  \maketitle
} \fi

\if1\blind
{
  \bigskip
  \bigskip
  \bigskip
  \begin{center}
    {\LARGE\bf Testing for Structural Breaks via Ordinal Pattern Dependence}
\end{center}
  \medskip
} \fi

\bigskip
\begin{abstract}
We propose new concepts in order to analyze and model the dependence structure between two time series. Our methods rely exclusively on the order structure of the data points. Hence, the methods are stable under monotone transformations of the time series and robust against small perturbations or measurement errors. Ordinal pattern dependence can be characterized by four parameters. We propose estimators for these parameters, and we calculate their asymptotic distributions. Furthermore, we derive a test for structural breaks within the dependence structure. All results are supplemented by simulation studies and empirical examples. 

For three consecutive data points attaining different values, there are six possibilities how their values can be ordered. These possibilities are called ordinal patterns. Our first idea is simply to count the number of coincidences of patterns in both time series, and to compare this with the expected number in the case of independence. If we detect a lot of coincident patterns, this means that the up-and-down behavior is similar. Hence, our concept can be seen as a way to measure non-linear `correlation'. We show in the last section, how to generalize the concept in order to capture various other kinds of dependence.
\end{abstract}

\noindent%
{\it Keywords:}  Time series, limit theorems, near epoch dependence, non-linear correlation.
%permutations, short range dependence, structural break, ordinal pattern dependence
\vfill

\newpage
\spacingset{1.45} % DON'T change the spacing!
\section{Introduction}
\label{sec:intro}

In Schnurr (2014) the concept of positive/negative ordinal pattern dependence has been introduced. In an empirical study he has found evidence that dependence of this kind appears in real-world financial data. In the present article, we provide consistent estimators for the key parameters in ordinal pattern dependence, and we derive their asymptotic distribution. Furthermore, we present a test for structural breaks in the dependence structure. The applicability of this test is emphasized both, by a simulation study as well as by a real world data example. Roughly speaking, positive (resp. negative) ordinal pattern dependence corresponds to a co-monotonic behavior (resp. an anti-monotonic behavior) of two time series. Sometimes an entirely different connection between time series might be given. By introducing certain distance functions on the space of ordinal patterns we get the flexibility to analyze various kinds of dependence. Within this more general framework we derive again limit theorems and a test for structural breaks.

Detecting changes in the dependence structure is an important issue in various areas of applications. Analyzing medical data, a change from a synchronous movement of two data sets to an asynchronous one might indicate a disease or e.g. a higher risk for a heart attack. In mathematical finance it is a typical strategy to diversify a portfolio in order to reduce the risk. This does only work, if the assets in the portfolio are not moving in the same direction all the time. Therefore, as soon as a strong co-movement is detected, it might be necessary to restructure the portfolio. 

From an abstract point of view, the objects under consideration are two discretely observed stochastic processes $(X_n)_{n\in\bbz}$ and $(Y_n)_{n\in\bbz}$. %They can be either a-priori discrete or they can be discrete observations of a model in continuous time.
%or even at random times $\sigma_0<\sigma_1<...<\sigma_n$ as long as the ordinal pattern probabilities can be calculated.
In order to keep the notation simple, we will always use $\bbz$ as index set. Increments are denoted by $(\Delta X_n)_{n\in\bbz}$, that is, $\Delta X_n:=X_n-X_{n-1}$. Furthermore, $h\in\bbn$ is the number of consecutive increments under consideration. %We will see below that classical time series models (like $ARMA$ processes) are contained in our theory. 

The dependence is modeled and analyzed in terms of so called `ordinal patterns'. At first we extract the ordinal information of each time series. With $h+1$ consecutive data points $x_0, x_1,...x_h$ (or random variables) we associate a permutation in the following way: we order the values top-to-bottom and write down the indices describing that order. If $h$ was four and we got the data $(x_0,x_1,x_2,x_3,x_4)= (2,4,1,7,3.5)$, the highest value is obtained at 3, the second highest at 1 and so on. We obtain the vector $(3,1,4,0,2)$ which carries the full ordinal information of the data points. This vector of indices is called the \emph{ordinal pattern} of $(x_0, ..., x_h)$. A mathematical definition of this concept is postponed to the subsequent section. There, it also becomes clear how to deal with coincident values within $(x_0,...,x_h)$. The reflected vector $(-x_0,...,-x_h)$ yields the inverse pattern, that is, read the permutation from right to left. In the next step we compare the probability (in model classes) respectively the relative frequency (in real data) of coincident patterns between the two time series. If the (estimated) probability of coincident patterns is much higher than it would be under the hypothetical case of independence, we say that the two time series admit a positive ordinal pattern dependence. In the context of negative dependence we analyze the appearance respectively the probability of reflected patterns. The degree of this dependence might change over time: we see below that structural breaks of this kind show up in the dependence between the S\&P 500 and its corresponding volatility index. 

Ordinal patterns have been introduced in order to analyze large noisy data sets which appear in neuro-science, medicine and finance (cf. Bandt and Pompe (2002), Keller et al. (2007), Sinn et al. (2013)). In all of these articles only a single data set has been considered. To our knowledge the present paper is the first approach to derive the technical framework in order to use ordinal patterns in the context of dependence structures and their structural breaks. 

The advantages of the method include that the analysis is stable under monotone transformations of the state space. The ordinal structure is not destroyed by small perturbations of the data or by measurement errors. Furthermore, there are quick algorithms to analyze the relative frequencies of ordinal patterns in given data sets (cf. Keller et al. (2007), Section 1.4). Reducing the complexity and having efficient algorithms at hand are important advantages in the context of Big Data. Furthermore, let us emphasize that unlike other concepts which are based on classical correlation, we do not have to impose the existence of second moments. This allows us to consider a bigger variety of model classes. 

The minimum assumption in order to carry out our analysis is that the time series under consideration are \emph{ordinal pattern stationary (of order h)}, that is, the probability for each pattern remains the same over time. In the sections on limit theorems we will have to be slightly more restrictive and have to impose stationarity of the underlying time-series. Obviously stationarity of a time series implies stationary increments, which in turn implies ordinal pattern stationarity.

The paper is organized as follows: in Section 2 we present the rigorous definitions of the concepts under consideration. In particular we recall and extend the concept of ordinal pattern dependence. For the reader's convenience we have decided to derive the test for structural breaks first for this classical setting. In order to show the applicability of the proposed test we consider financial index data. It is then a relatively simple task to generalize our results to the more general framework which is described in Section 3. There, we consider the new concept of average weighted ordinal pattern dependence. Some technical proofs have been postponed to Section 4. In Section 5 we present a short conclusion. 

From the practical point of view, our main results are the tests on structural breaks (cf. Theorem 2.7 and its corollary) and the generalization of the concept of ordinal pattern dependence (Section 3). In the theoretical part the limit theorems for all parameters under consideration, in particular for $p$, are most remarkable (cf. Corollary 2.6). 

The notation we are using is mostly standard: 
vectors are column vectors and $'$ denotes a transposed vector or matrix. In defining new objects we write `:=' where the object to be defined stands on the left-hand side. We write $\bbr_+$ for $[0,\infty)$.

%%%%%%%%%%%%%%%%%%%%%%%%%%%%%%%%%%%%%%%%%%%%%%%%%%%%%%%%%%%%%%%%%%%%%%%%%%%%%%%%%%%%%%%%%%%%%%%%
\section{Methodology}

First we fix some notations and the basic setup. Afterwards we present limit theorems for the parameters under consideration as well as our test on structural breaks. 

\subsection{Definitions and General Framework}

Let us begin with the formal definition of ordinal patterns: let $h\in\bbn$ and $\mathbf{x}=(x_0,x_1,...,x_h)\in\bbr^{h+1}$. The \emph{ordinal pattern} of $\mathbf{x}$ is the unique permutation $\Pi(\mathbf{x})=(r_0,r_1,...,r_h)\in S_{h+1}$ such that \\
\hspace*{2mm} (i) $x_{r_0} \geq x_{r_1} \geq ... \geq x_{r_h}$ and \\
\hspace*{2mm} (ii) $r_{j-1} > r_j$ if $x_{r_{j-1}} =x_{r_j}$ for $j\in\{1,...,h\}$.\\
For an element $\pi\in S_{h+1}$, $m(\pi)$ is the \emph{reflected permutation}, that is, read the permutation from right to left.

Let us now introduce the main quantities under consideration:\\
\hspace*{5mm}$p:=P \Big(\Pi(X_n,X_{n+1},...,X_{n+h})=\Pi(Y_n,Y_{n+1},...,Y_{n+h})\Big)$\\
\hspace*{5mm}$q:=\sum_{\pi\in S_{h+1}} P\Big(\Pi(X_n,X_{n+1},...,X_{n+h})=\pi\Big) \cdot P\Big(\Pi(Y_n,Y_{n+1},...,Y_{n+h})=\pi\Big)$\\
\hspace*{5mm}$r:=P \Big(\Pi(X_n,X_{n+1},...,X_{n+h})=m\big(\Pi(Y_n,Y_{n+1},...,Y_{n+h})\big)\Big) $\\
\hspace*{5mm}$s:=\sum_{\pi\in S_{h+1}} P\Big(\Pi(X_n,X_{n+1},...,X_{n+h})=\pi\Big) \cdot P\Big(m(\Pi(Y_n,Y_{n+1},...,Y_{n+h}))=\pi\Big)$

The time series $X$ and $Y$ exhibit a positive ordinal pattern dependence (ord$\oplus$) of order $h\in\bbn$ and level $\alpha>0$ if
\[
p>\alpha +q
\]
and negative ordinal pattern dependence (ord$\ominus$) of order $h\in\bbn$ and level $\beta>0$ if
\[
r>\beta+s.
\]

Let us shortly comment on the intuition behind these concepts: we compare the probability of coincident (resp. reflected) patterns in the time series $\{p,r\}$ with the (hypothetical) case of independence $\{q,s\}$. In order to have a concept which is comparable to correlation and other notions which describe or measure dependence between time series, we introduce the following quantity
\begin{align} \label{normed}
\text{ord}(X,Y):= \left(\frac{p-q}{1-q}\right)^+ - \left(\frac{r-s}{1-s}\right)^+
\end{align}
which is called the \emph{standardized ordinal pattern coefficient}. It has the following advantages: we obtain values between -1 and 1, becoming -1 resp. 1 in appropriate cases: let $Y$ be a monotone transformation of $X$ where $X$ is a time series which admits at least two different patterns with positive probability. In this case
\[
  \text{ord}(X,Y)= \left(\frac{1-q}{1-q}\right)^+ - \left(\frac{0-s}{1-s}\right)^+ =1 \ (q,s<1).
\]
In general $q$ becomes 1, only if the time series $X$ and $Y$ both admit only one pattern $\pi$ with positive probability (which is then automatically 1). In this case we would set $\text{ord}(X,Y)=1$, since this situation corresponds to a perfect co-movement. A similar statement holds true for $s$ in the case of anti-monotonic behavior. 

Using the standardized coefficient, the interesting parameters are still $p$ and $r$. If the time series $X$ and $Y$ under consideration are stationary, $q$ and $s$ do not change over time also. Recall that we do not want to find structural breaks within one of the time series, but in their dependence structure. In the context of change-points respectively structural breaks within \emph{one} data set cf. Sinn et al. (2012).

\begin{remark} It is important to note that our method depends on the definition of ordinal patterns which is not unique in the literature. In each case permutations are used in order to describe the relative position of $h+1$ consecutive data points. Most of the time the definition which we have given above is used. In Sinn et al. (2012), however, time is inverted while Bandt and Shiha (2007) use an entirely different approach which they call `order patterns'. Using their definition, the reflected pattern is no longer derived by reading the original pattern $\sigma$ from the right to the left, but by subtracting: $(h+1,...,h+1)-\sigma$.  However, the quantities $p$ and $q$ are invariant under bijective transformations (that is: renaming) of the ordinal patterns. Therefore, our results remain valid whichever definition is used.
\end{remark}

Given the observations $(x_1,y_1),\ldots, (x_n,y_n)$, we want to estimate the parameters $p,q,r,s$, and to test for structural breaks in the level of ordinal pattern dependence. In the subsequent section, we will propose estimators and test statistics, and determine their asymptotic distribution, as $n$ tends to infinity. Readers who are only interested in the test for structural breaks and its applications might skip the next subsection.

\subsection{Asymptotic Distribution of the Estimators of $p$}

The natural estimator of the parameter $p$ is the sample analogue
\begin{equation}
\hat{p}_n=\frac{1}{n} \sum_{i=1}^{n-h} 1_{\{\Pi(X_i,\ldots,X_{i+h})=\Pi(Y_i,\ldots,Y_{i+h})   \}}.
\label{eq:p_hat}
\end{equation}

The asymptotic results in our paper require some assumptions regarding the dependence structure of the underlying process $(X_i,Y_i)_{i\in\bbz}$. Roughly speaking, our results hold if the process is `short range dependent'. Specifically, we will assume that $(X_i,Y_i)_{i\in \bbz}$ is a functional of an absolutely regular process.  This assumption is valid for many processes arising in probability theory, statistics and analysis; see e.g. Borovkova, Burton and Dehling (2001)  for a large class of examples. 

For the reader's convenience we recall the following concept: let $(\Omega,\cf,P)$ be a probability space.  Given two sub-$\sigma$-fields $\ca, \cb\subset \cf$, we define 
\[
 \beta(\ca,\cb)=\sup \sum_{i,j} |P(A_i\cap B_j)-P(A_i)P(B_j)|,
\]
where the $\sup$ is taken over all partitions $A_1,\ldots, A_I\in \ca$ of $\Omega$, and over all partitions $B_1,\ldots,B_J\in \cb$ of $\Omega$.
The stochastic process $(Z_i)_{i\in \bbz}$ is called \emph{absolutely regular} with coefficients $(\beta_m)_{m\geq 1}$, if 
\[
  \beta_m:= \sup_{n\in \bbz} \beta(\cf_{-\infty}^n,\cf_{n+m+1}^\infty) \rightarrow 0,
\]
as $m\rightarrow \infty$. Here $\cf_k^l$ denotes the $\sigma$-field generated by the random variables 
$Z_k,\ldots,Z_l$. 

Now we can state our main assumption. We will see below that it is very weak and that the class under consideration contains several interesting and relevant examples. 

Let $(X_i,Y_i)_{i\geq 1}$ be an $\bbr^2$-valued stationary process, and let $(Z_i)_{i\in \bbz}$ be a stationary process with values in some measurable space $S$.  We say that $(X_i,Y_i)_{i\geq 1}$ is a functional of the process $(Z_i)_{i\in \bbz}$, if there exists a measurable function $f:S^\bbz\rightarrow \bbr^2$ such that, for all $k\geq 1$,
\[
  (X_k,Y_k)=f((Z_{k+i})_{i\in \bbz}).
\]
We call $(X_i,Y_i)_{i\geq 1}$ a $1$-approximating functional with constants $(a_m)_{m\geq 1}$, if for any $m\geq 1$, there exists a function $f_m: S^{2m+1}\rightarrow \bbr^2$ such that (for every $i\in\bbz$)
\begin{equation}
 E\|(X_i,Y_i)-f_m(Z_{i-m},\ldots,Z_{i+m})\| \leq a_m.
\label{eq:1-appr}
\end{equation}
\label{def:far}

Note that, in the Econometrics literature $1$-approximating functionals are called $L_1$-near epoch dependent (NED).
The following examples show the richness of the class under consideration. Recall that every causal $ARMA(p,q)$ process can be written as an $MA(\infty)$ process (cf. Brockwell and Davis (1991) Example 3.2.3.).

\begin{example} \label{ex:oneapprox}
(i) Let $(X_i)_{i\geq 1}$ be an $MA(\infty)$ process, that is,
\[
  X_i=\sum_{j=0}^\infty \alpha_j Z_{i-j}
\]
where $(\alpha_j)_{j\geq 0}$ are real-valued coefficients with $\sum_{i=j}^\infty \alpha_j^2<\infty$, and where $(Z_i)_{i\in \bbz}$ is an i.i.d. process with mean zero and finite variance. $(X_i)_{i\geq 1}$ is a $1$-approximating functional with coefficients 
$a_m=\left(\sum_{j=m+1}^\infty \alpha_j^2  \right)^{1/2} $. Limit theorems for $MA(\infty)$ processes require that the sequence $(a_m)_{m\geq 0}$ decreases to zero sufficiently fast. The minimal requirement is usually that the coefficients $(\alpha_j)_{j\geq 0}$ are absolutely summable.  If this condition is violated, the process may exhibit long range dependence, which is e.g. characterized by non-normal limits and by a scaling different from the usual $\sqrt{n}$-scaling. Let us remark that ordinal pattern distributions in (a single) $ARMA$ time series have been investigated in Bandt and Shiha (2007) Section 6.
\\[1mm]
(ii) Consider the map $T:[0,1]\longrightarrow [0,1]$, defined by $T(\omega)=2\, \omega\mod 1$, i.e.,
\[
 T(\omega) = \left\{   \begin{array}{ll}
2\omega & \mbox{ if } 0\leq \omega\leq 1/2 \\
2\omega-1& \mbox{ if } 1/2< \omega \leq 1.
\end{array}
\right.
\]
This function is well known as the one-dimensional baker's map in the theory of dynamical systems.
Let $g:[0,1]\rightarrow \bbr$ be a Lipschitz-continuous function, and define the stochastic process
$(X_n)_{n\geq 0}$ by 
\[
  X_n(\omega)=g(T^n(\omega)).
\]
This process was studied by Kac (1946), who established the central limit theorem for partial sums 
$\sum_{i=1}^n X_i$, under the assumption that $g$ is a function of bounded variation. The time series $(X_n)_{n\geq 0}$ is a $1$-approximating functional of an i.i.d. process $(Z_j)_{j\in \bbz}$ with approximating constants $a_m={\|g\|_L}/{2^{m+1}}$ where $\norm{\cdot}_L$ denotes the Lipschitz norm.
\\[1mm]
(iii) The continued fraction expansion provides an example from analysis  that falls under the framework of the processes studied in this paper. It is well known that any $\omega\in (0,1]$ has a unique continued fraction expansion
\[
  \omega=\frac{1}{a_1+\frac{1}{a_2+\frac{1}{a_3 +\cdots}}},
\] 
where the coefficients $a_i$, $i\geq 1$, are non-negative integers. Since these coefficients are functions of $\omega$, we obtain a stochastic process $(Z_i)_{i\geq 1}$, defined on the probability space $\Omega=(0,1]$ by $Z_i(\omega)=a_i$. If we equip $(0,1]$ with the Gau\ss\ measure
\[
  \mu((0,x])=\frac{1}{\log 2} \log(1+x),
\]
the process $(Z_i)_{i\geq 1}$ becomes a stationary $\psi$-mixing process. We can then study the process of remainders 
\[
  X_n(\omega)=\frac{1}{Z_n(\omega)+\frac{1}{Z_{n+1}(\omega)+\frac{1}{Z_{n+2}(\omega) +\cdots}}}.
\]
The process $(X_n)_{n\geq 1}$ is a $1$-approximating functional of the process $(Z_i)_{i\geq 1}$, and thus the results of the present paper are applicable to this example.
\end{example}

\begin{remark} At first glance, it might be a bit surprising that examples from the theory of dynamical systems are treated in an article which deals with the order structure of data. In fact, there is a close relationship between these two mathematical subjects: 
using ordinal patterns in the analysis of time series is equivalent to dividing the state-space into a finite number of pieces and using only the information in which piece the state is contained at a certain time. This is known as symbolic dynamics in the theory of dynamical systems. Each of these pieces is assigned with a so called symbol. Hence, orbits of the dynamical system are turned into sequences of symbols (cf. Keller et al. (2007), Section 1.2). 
\end{remark}

Processes that are $1$-approximating functionals of an absolutely regular process satisfy practically all limit theorems of probability theory, provided the $1$-approximation coefficients $a_m$ and the absolute regularity coefficients $\beta_k$ decrease sufficiently fast. In our applications below, we are not so much interested in limit theorems for the $(X_i,Y_i)$-process itself, but in limit theorems for certain functions
$g((X_i,Y_i),\ldots,(X_{i+h},Y_{i+h}))$ of the data. We then have to show that these functions are $1$-approximating functionals, as well. We will now state this result for two functions that play a role in the context of the present research. A preliminary lemma, along with its proof, is postponed to Section 4. 

\begin{theorem}
Let $(X_i,Y_i)_{i\geq 1}$ be a stationary 1-approximating functional of the absolutely regular process $(Z_i)_{i\geq 1}$. Let $(\beta(k))_{k \geq 1}$ denote the  mixing coefficients of the process $(Z_i)_{i\geq 1}$, and let $(a_k)_{k\geq 1}$ denote the 1-approximation constants. Assume that
\begin{equation}
 \sum_{k=1}^\infty (\sqrt{a_k}+\beta(k)) <\infty.
\label{eq:rates}
\end{equation}
Furthermore, assume that the distribution functions of $X_i-X_1$, and of $Y_i-Y_1$, are both Lipschitz-continuous, for any $i\in \{1,\ldots,h+1\}$. 
Then, as $n\rightarrow \infty$,
\begin{equation}
 \sqrt{n} (\hat{p}_n-p) \stackrel{\mathcal{D}}{\longrightarrow}N(0,\sigma^2),
\label{eq:p_hat_ad}
\end{equation}
where the asymptotic variance is given by the series
\begin{eqnarray}
\sigma^2 &=& \Var(1_{\{ \Pi(X_1,\ldots,X_{h+1})=\Pi(Y_1,\ldots,Y_{h+1})  \}}) \label{eq:p_limvar}\\
 && \qquad+2\sum_{m=2}^\infty \Cov\left(1_{\{ \Pi(X_1,\ldots,X_{h+1})=\Pi(Y_1,\ldots,Y_{h+1}) \}},
 1_{\{ \Pi(X_m,\ldots,X_{m+h})=\Pi(Y_m,\ldots,Y_{m+h}) \}}\right),
\nonumber
\end{eqnarray}
\label{th:p_hat_ad}
\end{theorem}
\begin{proof}
We apply Theorem~18.6.3 of Ibragimov and Linnik (1971) to the partial sums of the random variables
\[
 \xi_i:=  1_{\{\Pi(X_i,\ldots,X_{i+h})=\Pi(Y_i,\ldots,Y_{i+h})   \}}.
\]
By Lemma~\ref{le:1-approx}, we get that $\xi_i$ is a 1-approximating functional of the process 
$(Z_i)_{i\geq 1}$ with approximation constants $(\sqrt{a_k})_{k\geq 1}$. Thus, the conditions of Theorem~18.6.3 of Ibragimov and Linnik (1971) are satisfied, and hence \eqref{eq:p_hat_ad} holds.
\end{proof}

\begin{remark} Theorem~\ref{th:p_hat_ad} holds under the assumption that the underlying time series
$(X_i,Y_i)_{i\geq 1}$ is short range dependent. In the case of long-range dependent time series, other limit theorems hold, albeit with a normalization that is different from the standard $\sqrt{n}$-normalization.
\end{remark}

In order to determine asymptotic confidence intervals for $p$ using the above limit theorem, we need to estimate the limit variance $\sigma^2$. De Jong and Davidson (2000) have proposed a kernel estimator for the series on the r.h.s.\ of \eqref{eq:p_limvar}. Let $k:\bbr\rightarrow [0,1]$ be a symmetric kernel, i.e. $k(-x)=k(x)$, that is continuous in $0$ and safisfies $k(0)=1$, and let $(b_n)_{n\geq 1}$ be a bandwidth sequence tending to infinity. 
Then we define the estimator
\begin{equation}
\hat{\sigma}_n^2=\frac{1}{n}\sum_{i=1}^{n-h}\sum_{j=1}^{n-h}
 k\!\left(\! \frac{i-j}{b_n}  \right)\!\!
\left(1_{\{\Pi(X_i,\ldots,X_{i+h})=\Pi(Y_i,\ldots,Y_{i+h}) \}}\! -\!\hat{p}_n \right)\!\!
\left(1_{\{\Pi(X_j,\ldots,X_{j+h})=\Pi(Y_j,\ldots,Y_{j+h}) \}}\! -\!\hat{p}_n \right)\! .
\label{eq:var_est}
\end{equation}
De Jong and Davidson (2000) show that $\hat{\sigma}_n^2$ is a consistent estimator of $\sigma^2$, provided some technical conditions concerning the kernel function $k$, the bandwidth sequence
$(b_n)_{n\geq 1}$ and the process $(X_i,Y_i)_{i\geq 1}$ hold. The assumptions on the process follow from our assumptions. Concerning the kernel function and the bandwidth sequence, a possible choice is given by $k(x)=(1-|x|)1_{[-1,1]}(x)$ and $b_n=\log(n)$. We thus obtain the following corollary to Theorem~\ref{th:p_hat_ad}.

\begin{corollary}
Under the same assumptions as in Theorem~\ref{th:p_hat_ad}
\[
  \frac{\sqrt{n}(\hat{p}_n-p)}{\hat{\sigma}_n} \stackrel{\mathcal{D}}{\longrightarrow} N(0,1).
\]
As a consequence, $[\hat{p}_n-z_\alpha \hat{\sigma}_n,\hat{p}_n-z_\alpha \hat{\sigma}_n]$ is a confidence interval
with asymptotic coverage probability $(1-\alpha)$. Here $z_\alpha$ denotes the upper $\alpha$ quantile of the standard normal distribution.
\end{corollary}

We complement this theoretical result with a simulation of two correlated standard normal $AR(1)$ time series, where the $AR$-parameter $\phi$ is 0.1. Furthermore we have set $h=2$, $p=0.6353$, $n=1000$, $k(x)$ and $b_n$ as above. We have simulated this 1000 times obtaining the following histogram and Q--Q plot. 

\begin{figure}[h]
\begin{minipage}[h]{0.47\textwidth}
\includegraphics[width=7cm, angle=0]{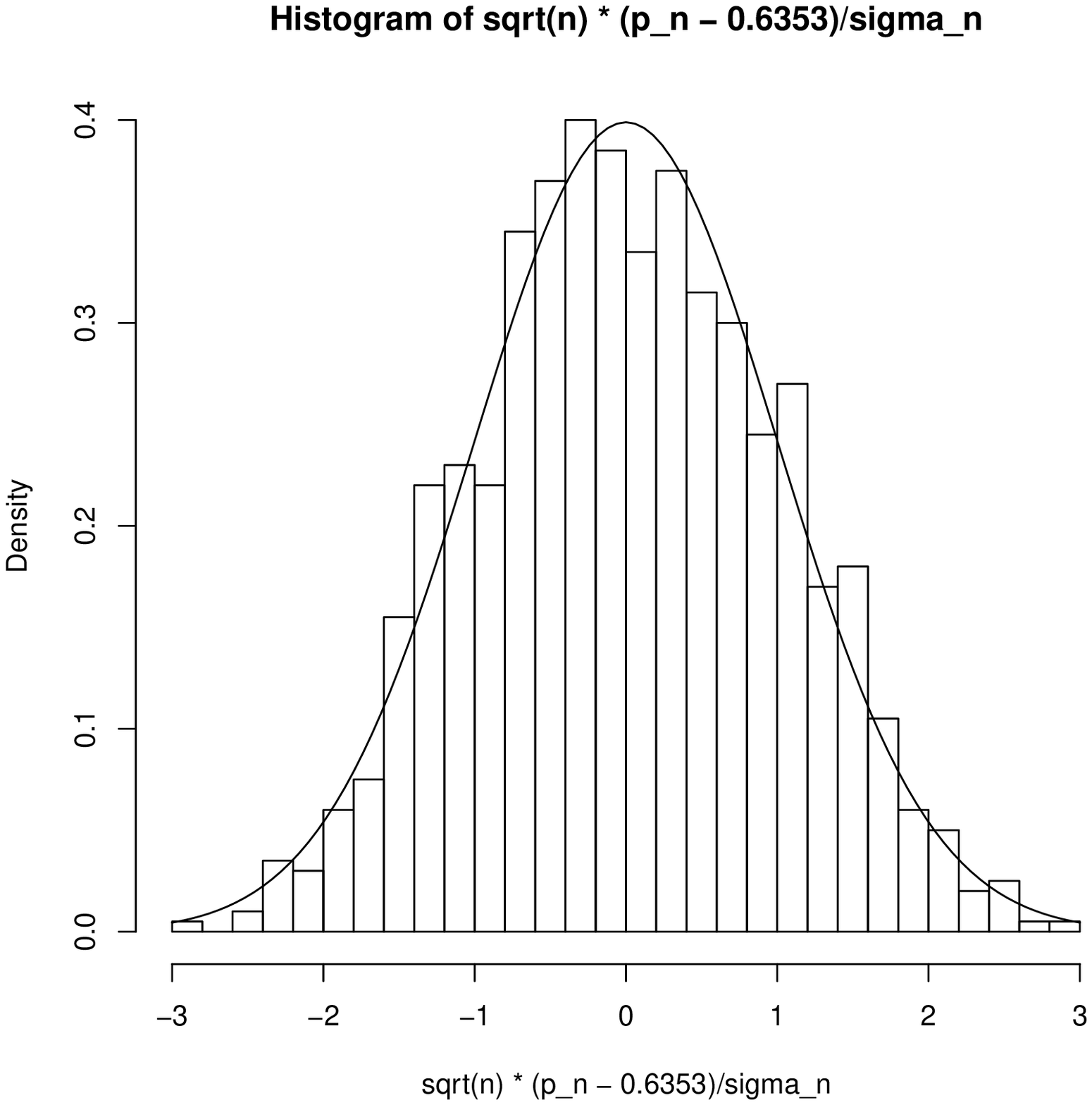}
\caption{Histogram of ${\sqrt{n}(\hat{p}_n-p)}/{\hat{\sigma}_n}$ for 1000 simulations of correlated $AR(1)$ time series and density of N(0,1) distribution.}
\end{minipage}
\hfill
\begin{minipage}[h]{0.47\textwidth}
\includegraphics[width=7cm, angle=0]{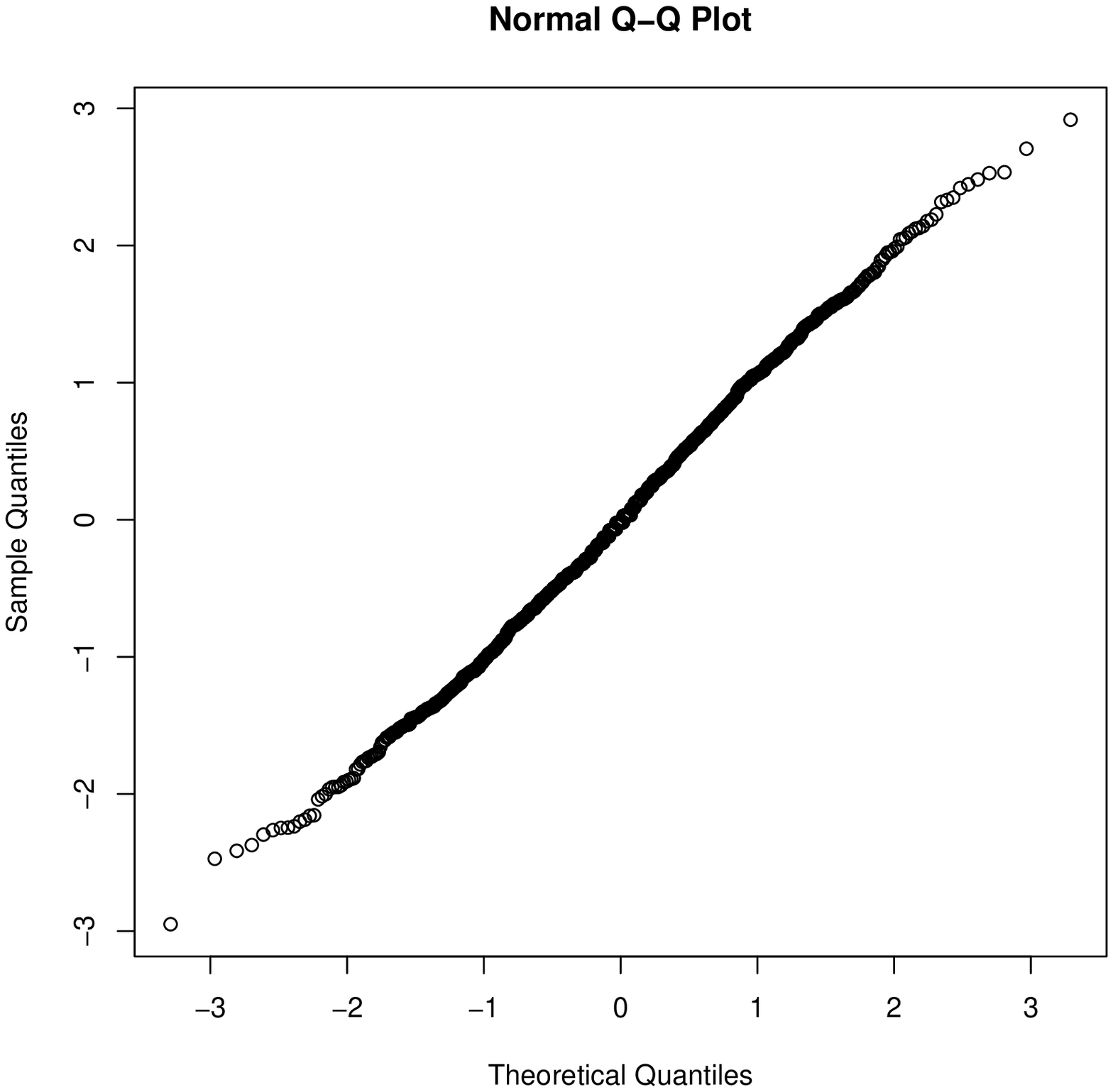}
\caption{Q--Q plot of ${\sqrt{n}(\hat{p}_n-p)}/{\hat{\sigma}_n}$ for 1000 simulations of correlated $AR(1)$ time series and N(0,1) distribution.}
\end{minipage}
\end{figure}

\subsection{Structural Breaks}

%[Hier w\"are zuerst zu entscheiden, welche der oben eingef\"uhrten Parameter hinsichtlich eines Strukturbruchs betrachtet werden sollten. Ich nehme jetzt erst den einfachsten Fall, und konzentriere mich auf den Parameter $p$. Die wesentliche Idee besteht dann darin, f\"ur jedes $k\in \{1,\ldots,n-h\}$ den Parameter $p$ aus den Daten $(X_1,Y_1),\ldots, (X_k,Y_k)$ zu sch\"atzen und diese Sch\"atzer dann jeweils mit $\hat{p}_n$ zu vergleichen. Die maximale Abweichung, noch geeignet zu normieren, ist dann eine sinnvolle Teststatistik. Die asymptotische Verteilung der Teststatistik unter der Hypothese kann mit Hilfe eines Invarianzprinzips f\"ur den Partialsummenprozess der Indikatorvariablen hergeleitet werden.]

As we have pointed  out above the interesting parameter, in order to detect structural breaks in the dependence structure, is $p$. If $p$ changes significantly over time, $r$ has to change also. Furthermore, in order to analyze $r$ one can instead analyze $p$ for $X$ and $-Y$. For stationary time series, the values of $q$ and $s$ are stationary over time, too.

In order to test the hypothesis that there is no change in the ordinal pattern dependence, we propose the test statistic
\begin{equation} \label{T_n}
 T_n=\max_{0\leq k\leq n-h} \frac{1}{\sqrt{n}} \left|\sum_{i=1}^k \left(1_{\{\Pi(X_i,\ldots,X_{i+h})=
\Pi(Y_i,\ldots,Y_{i+h})  \}} -\hat{p}_n\right)\right|
\end{equation}
and prove limit theorems which are valid under the hypothesis. 
\begin{theorem} \label{th:cp_test_ad}
Under the same assumptions as in Theorem~\ref{th:p_hat_ad}, we have
\[
  T_n\stackrel{\mathcal{D}}{\longrightarrow} \sigma \sup_{0\leq \lambda \leq 1} 
  |W(\lambda)-\lambda W(1)|,
  \]
where $\sigma$ is defined in \eqref{eq:p_limvar}, and where $(W(\lambda))_{0\leq \lambda \leq 1}$ is standard Brownian motion. 
\end{theorem}

The proof is again postponed to Section 4. 

\begin{corollary}
Under the same assumptions as in Theorem~\ref{th:p_hat_ad}, we have
\[
  \frac{1}{\hat{\sigma}_n} T_n\stackrel{\mathcal{D}}{\longrightarrow} \sup_{0\leq \lambda \leq 1} 
  |W(\lambda)-\lambda W(1)|,
  \]
where $\hat{\sigma}_n$ is defined in \eqref{eq:var_est}, and where $(W(\lambda))_{0\leq \lambda \leq 1}$ is standard Brownian motion. 
\end{corollary}
\begin{remark}{\rm
Note that the distribution of $\sup_{0\leq \lambda \leq 1}   |W(\lambda)-\lambda W(1)|$ is the Kolmogorov distribution. If we denote the upper $\alpha$ quantile of the Kolmogorov distribution by $k_{\alpha}$, the test that rejects the hypothesis of no change when ${T_n}/{\hat{\sigma}_n}  \geq k_\alpha$ has level $\alpha$.
}
\end{remark}

\begin{example}

Again we complement the theoretical result with a simulation study. In both cases we have simulated two correlated $AR(1)$ time series, where the $AR$-parameter $\phi$ is 0.2. Furthermore we have set $h=2$, $n=1000$, $k(x)$ and $b_n$ as above. In the first time series (Figure 3), there is no structural break ($p=0.6353$). In the second time series (Figure 4) we have set $p=0.6353$ for the first 500 data points and $p=0.5378$ for the second 500 data points. We have simulated both pairs of time series 1000 times. Recall that the 0.95 quantile of the Kolmogorov distribution is 1.36.

\begin{figure}[h!]
\begin{minipage}[h!]{0.47\textwidth}
\includegraphics[width=7cm, angle=0]{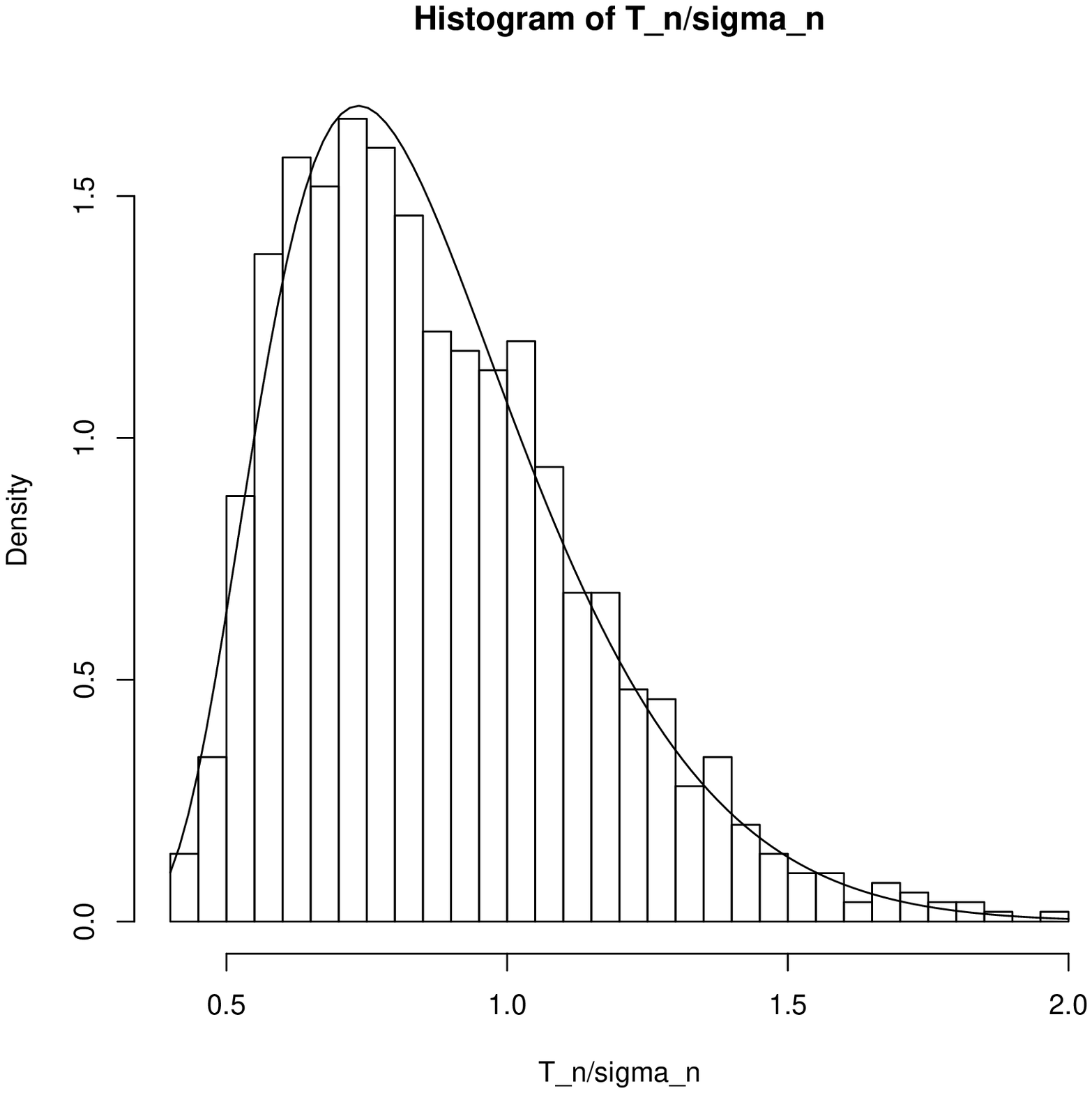}
\caption{Histogram of ${T_n}/{\hat{\sigma}_n}$ for 1000 simulations of correlated $AR(1)$ time series without structural break and Kolmogorov density.}
\end{minipage}
\hfill
\begin{minipage}[h!]{0.47\textwidth}
\includegraphics[width=7cm, angle=0]{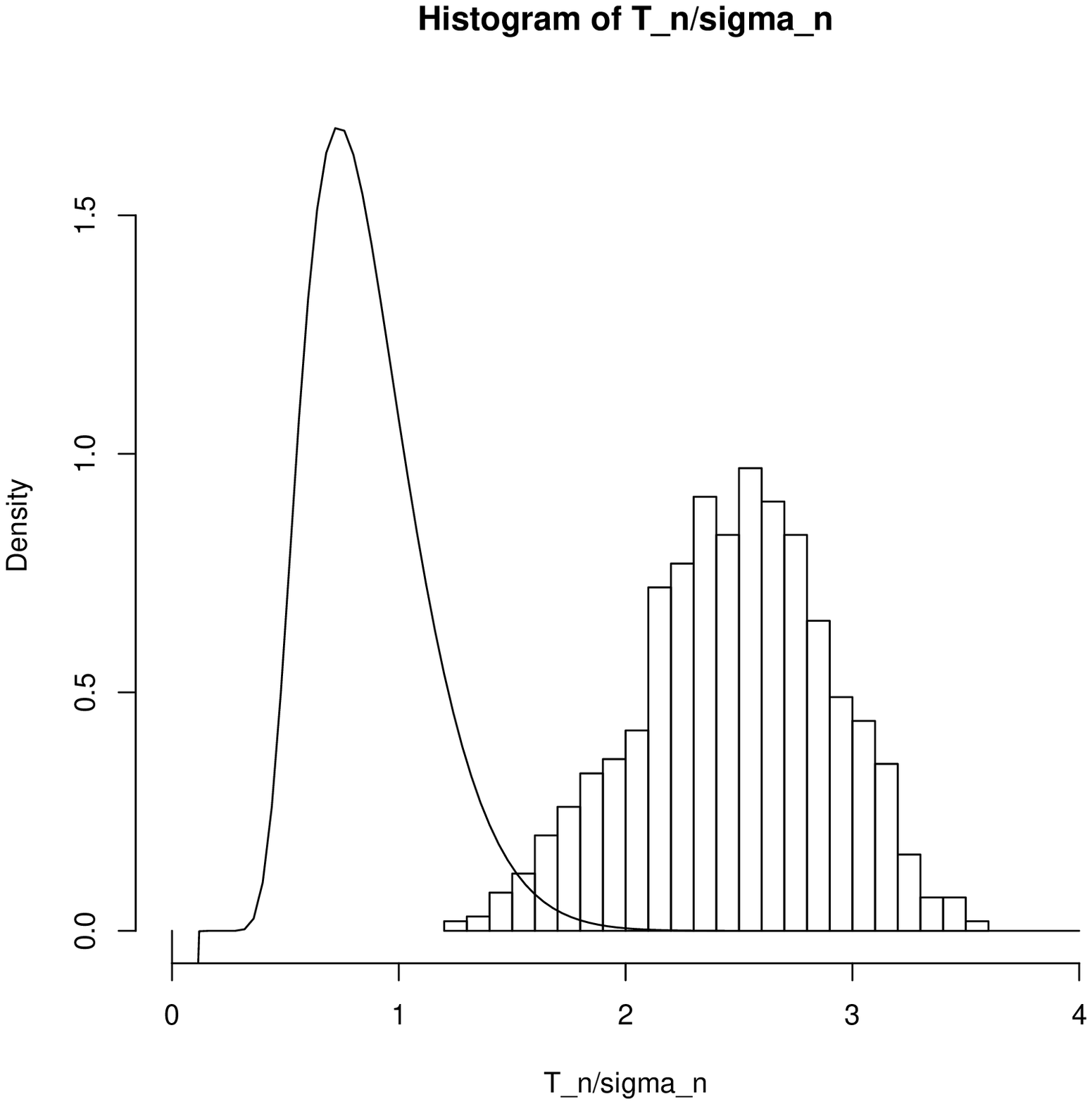}
\caption{Histogram of ${T_n}/{\hat{\sigma}_n}$ for 1000 simulations of correlated $AR(1)$ time series with a change of $p$ after 500 observations and Kolmogorov density}
\end{minipage}
\end{figure}
\end{example}

Let us, furthermore, analyze empirically the power of the test under different sizes of the change: we have used the same setting as above and obtain in Figure 5 the power of the test for various values of p after the break. 

\begin{figure}[H]
\centerline{\includegraphics[width=7cm, angle=0]{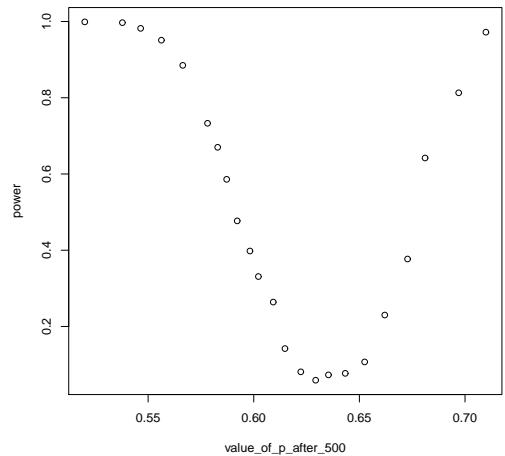}}
\caption{Empirical power of the test for structural breaks for different values of $p$ after 500 data points.}  
\end{figure}

We now work with simulated data sets which have been generated under three distinct settings: the normal distribution, the Student-t distribution with 2 degrees of freedom and the Cauchy distribution. For 500, 1000 and 2000 data points, we analyze structural breaks after 1/4, 1/3 respectively 1/2 of the data.   
For the resulting $AR(1)$-time series with $\phi=0.2$ as above we analyze changes in $p$ from $0.635$ to $0.437$.

\tiny
\vspace{5mm}
\centerline{
    \begin{tabular}{c c c c c c c c c c c c c c} 
     \hline \hline
				&        & \hspace*{-10mm}n=500     &      &&      &         & \hspace*{-10mm}n=1000&             &&     &  & \hspace*{-10mm}n=2000 &         \\ \cline{1-4} \cline{6-9} \cline{11-14}
		break & Normal & Student-t & Cauchy && break & Normal & Student-t & Cauchy && break & Normal & Student-t & Cauchy\\
     125  & 0.628 & 0.611  & 0.559 && 250 & 0.938  & 0.891 & 0.861 && 500 & 0.998 & 0.998  & 0.996 \\
		 167  & 0.776 & 0.769  & 0.71 && 333 & 0.979  & 0.973  & 0.958 && 667 & 1 & 0.999 & 1 \\
		 250  & 0.877 & 0.851  & 0.81 && 500 & 0.997  & 0.992 & 0.984 && 1000 & 1 & 1 & 1 
    \end{tabular}
    } 
\vspace{5mm}
\normalsize

Let us emphasize that in medical and financial data n=2000 is a reasonable number which is often obtained. It is surprising that we get strong results even in the highly irregular Cauchy setting.  

Finally, we use our method on real data.

\begin{example} \label{ex:vix} Let us consider the S\&P 500 and its corresponding volatility Index VIX. We cannot go into the details of the Chicago Board Options Exchange Volatility Index (VIX), but we give a short overview: the index was introduced in 1993 in order to measure the US-market's expectation of 30-day volatility which is implied by at-the-money S\&P 100 option prices. Since 2003, the VIX is calculated based on S\&P 500 data (we write SPX for short). The VIX is often qualified as the `fear index' in newspapers, TV shows and also in research papers. It has been discussed whether the VIX is a self-fulfilling prophecy or if it is a good predictor for the future anyway (cf. Whaley (2008), and the references given therein). For us the following facts are of importance:
\begin{itemize}
\item The VIX can be used to measure the market volatility at the moment it is calculated (instead of trying to predict the future).
\item Whether we use the S\&P 100 or the S\&P 500 data makes no difference, they are `for all intents and purposes (...) perfect substitutes' (Whaley (2008), p.3).
\item The relation between the two datasets (SPX$\leftrightarrow$VIX) is difficult to model (cf. Madan and Yor (2011)).
\item There is a negative relation between the datasets which is asymmetric and hence, in particular, not linear (cf. Whaley (2008), Section IV).
\end{itemize}
We have used open source data which we have extracted from finance.yahoo.com. We have analyzed the daily `close prices' for two periods of time each consisting in 2000 data points for $h=2$. In the time period from 1990-01-02 (the first day for which the VIX has been calculated) until 1997-11-25 we obtain $T_{2000}/\hat{\sigma}_{2000}= 0.843$. In the time period from 1997-11-26 to 2005-11-08 we get $T_{2000}/\hat{\sigma}_{2000}=1.5174$. Hence, our test suggests that there has been a structural break in the dependence between the two time series in this second time period (level $\alpha=0.05$). Recall that the so called dot-com bubble falls in the second time period. The effect gets weaker as $h$ increases. However, for $h=3$ resp. $h=4$ we still get significant results in case of the second time period, namely,  $T_{2000}/\hat{\sigma}_{2000}$ is 1.4898 resp. 1.3616.
Let us have a closer look on the values of
\[
 \frac{1}{\sqrt{n} \hat{\sigma}_n} \left|\sum_{i=1}^k \left(1_{\{\Pi(X_i,\ldots,X_{i+h})=
\Pi(Y_i,\ldots,Y_{i+h})  \}} -\hat{p}_n\right)\right|
\]
before the maximum in \eqref{T_n} is taken. The vertical line is the 0.95-quantile of the Kolmogorov distribution. 

\begin{figure}[H]
\begin{minipage}[h!]{0.47\textwidth}
\includegraphics[width=7cm, angle=0]{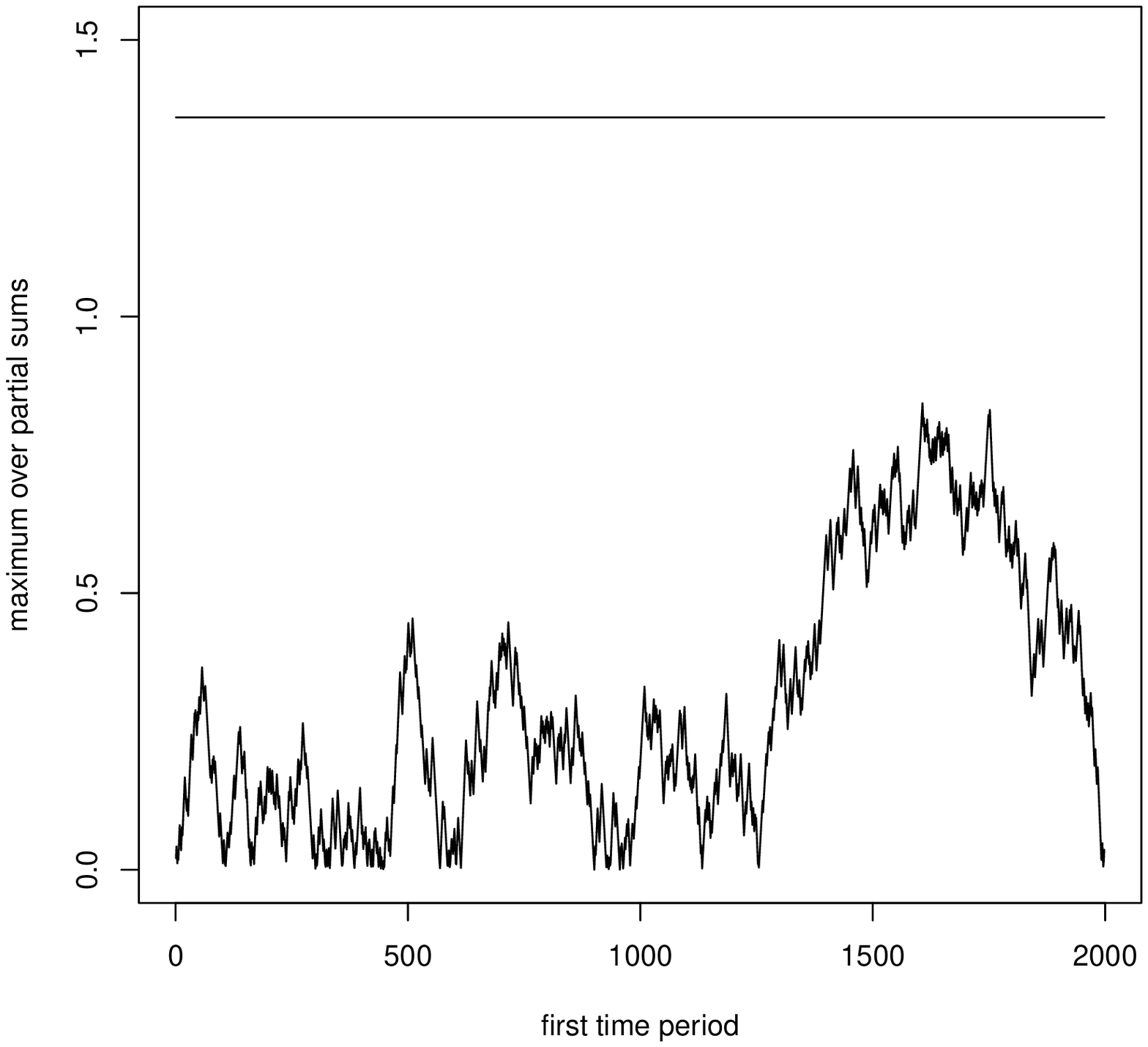}
\caption{No structural break is detected in the first 8 year period.}
\end{minipage}
\hfill
\begin{minipage}[h!]{0.47\textwidth}
\includegraphics[width=7cm, angle=0]{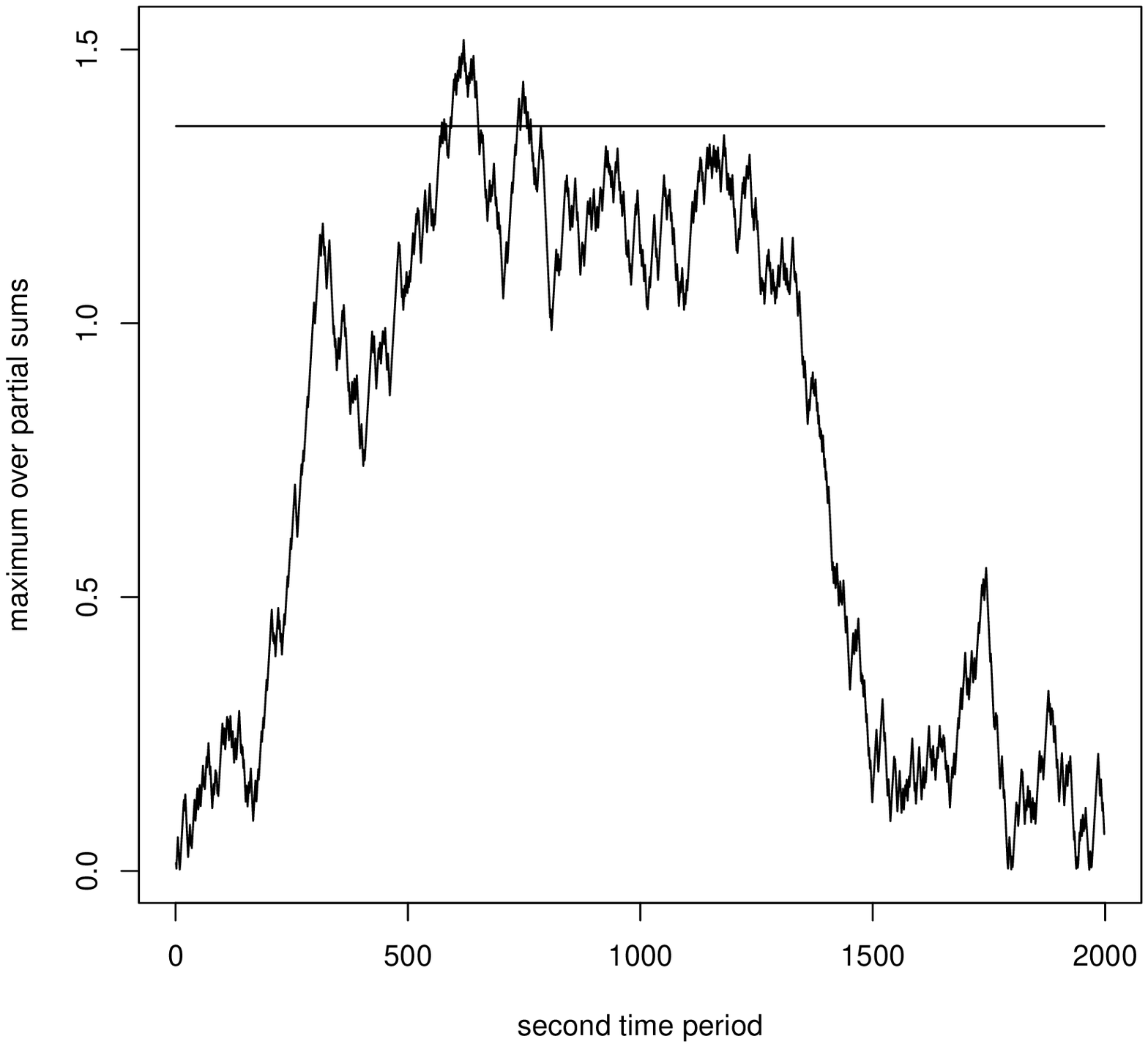}
\caption{In the second time period a structural break is detected.}
\end{minipage}
\end{figure}

\end{example}

\subsection{Estimating the Other Parameters}

Now we deal with the other parameters under consideration in order to estimate the standardized ordinal pattern coefficient. 

To estimate the parameter $q$, we define the following auxiliary parameters
\begin{eqnarray}
q_X(\pi) &=& P(\Pi(X_1,\ldots,X_{h})=\pi)\label{eq:q_x} \\
q_Y(\pi)&=&  P(\Pi(Y_1,\ldots,Y_{h})=\pi). \label{eq:q_y}
\end{eqnarray}
where $\pi\in S_{h+1}$ denotes a permutation.
Note that we have the following identity
\[
 q=\sum_{\pi\in S_h} q_X(\pi)\, q_Y(\pi).
\]
We estimate the parameters $q_X(\pi)$ and $q_Y(\pi)$ by their sample analogues
\begin{eqnarray}
\hat{q}_X(\pi)&=& \frac{1}{n}  \sum_{i=1}^{n-h} 1_{\{\pi(X_i,\ldots,X_{i+h})=\pi \}}\label{eq:q_x_hat} \\
\hat{q}_Y(\pi)&=& \frac{1}{n}  \sum_{i=1}^{n-h} 1_{\{\pi(Y_i,\ldots,Y_{i+h})=\pi \}}, \label{eq:q_y_hat} 
\end{eqnarray}
and finally $q$ by the plug-in estimator
\begin{equation}
 \hat{q}_n=\sum_{\pi\in S_h} \hat{q}_X(\pi)\, \hat{q}_Y(\pi).
\label{eq:q_n}
\end{equation}

\begin{theorem}
Under the same conditions as in Theorem~\ref{th:p_hat_ad}, the random vector
\[
  \sqrt{n} \left((\hat{q}_X(\pi) -q_X(\pi))_{\pi\in S_h},(\hat{q}_Y(\pi) -q_Y(\pi))_{\pi\in S_h}\right)
\]
converges in distribution to a multivariate normal distribution with mean vector zero and covariance matrix 
\[
\mathbf{\Sigma} =\left(  
\begin{array}{cc}
\mathbf{\Sigma}_{11} & \mathbf{\Sigma}_{12} \\
\mathbf{\Sigma}_{12} & \mathbf{\Sigma}_{22}
\end{array}
\right),
\] 
where the entries of the $(n!\times n!)$ block matrices  
$\mathbf{\Sigma}_{11} =\left(\sigma^{11}_{\pi,\pi^\prime}\right)_{\pi,\pi^\prime \in S_h}$ ,
$\mathbf{\Sigma}_{12} =\left(\sigma^{12}_{\pi,\pi^\prime}\right)_{\pi,\pi^\prime \in S_h}$, and 
 $\mathbf{\Sigma}_{22}=\left(\sigma^{22}_{\pi,\pi^\prime}\right)_{\pi,\pi^\prime \in S_h}$ are given by the following formulas 
\begin{eqnarray}
 \sigma_{\pi ,\pi^\prime}^{11} &=& \sum_{k=-\infty}^\infty
\Cov(1_{\{ \pi(X_1,\ldots,X_h)=\pi  \} },  1_{\{ \pi(X_{k+1},\ldots,X_{k+h})=\pi^\prime  \} } ) 
\label{eq:sig_11}\\
 \sigma_{\pi ,\pi^\prime}^{12} &=& \sum_{k=-\infty}^\infty
\Cov(1_{\{ \pi(X_1,\ldots,X_h)=\pi  \} },  1_{\{ \pi(Y_{k+1},\ldots,Y_{k+h})=\pi^\prime  \} } ) 
\label{eq:sig_12}\\
 \sigma_{\pi ,\pi^\prime}^{22} &=& \sum_{k=-\infty}^\infty
\Cov(1_{\{ \pi(Y_1,\ldots,Y_h)=\pi  \} },  1_{\{ \pi(Y_{k+1},\ldots,Y_{k+h})=\pi^\prime  \} } ).
\label{eq:sig_22}
\end{eqnarray}
\label{th:q_XY}
\end{theorem}
\begin{proof}
This follows from the multivariate  CLT for  functionals of mixing processes, which can be derived from Theorem~18.6.3 of Ibragimov and Linnik (1971) by using the Cram\'er-Wold device.
\end{proof}
\begin{remark}  
We have presented the formulas \eqref{eq:sig_11}--\eqref{eq:sig_22} for the asymptotic covariances for the case when the underlying process $(X_k,Y_k)_{k\in \bbz}$ is two-sided. In the case of a one-sided process
$(X_k,Y_k)_{k\geq 1}$, the formulas have to be adapted. E.g., in this case \eqref{eq:sig_11} becomes
\begin{eqnarray*}
\sigma_{\pi\pi^\prime}^{11}&=&
\Cov( 1_{\{ \pi(X_1,\ldots,X_h)=\pi  \} },  1_{\{ \pi(X_{1},\ldots,X_{h})=\pi^\prime  \} } ) \\
&&\qquad +\sum_{k=1}^\infty
\Cov(1_{\{ \pi(X_1,\ldots,X_h)=\pi  \} },  1_{\{ \pi(X_{k+1},\ldots,X_{k+h})=\pi^\prime  \} } ) \\
&& \qquad + \sum_{k=1}^\infty
\Cov(1_{\{ \pi(X_1,\ldots,X_h)=\pi^\prime  \} },  1_{\{ \pi(X_{k+1},\ldots,X_{k+h})=\pi \} } ) 
\end{eqnarray*}
\end{remark}

Using Theorem~\ref{th:q_XY} and the delta method, we can now derive the asymptotic distribution of the estimator $\hat{q}_n$, defined in \eqref{eq:q_n}. The proof can be found in Section 4. 

\begin{theorem} \label{th:q_hat}
Under the same assumptions as in Theorem~\ref{th:p_hat_ad}, 
\[
  \sqrt{n} (\hat{q}_n -q)  \rightarrow N(0,\gamma^2),
\] 
where the asymptotic variance $\gamma^2$ is given by the formula
\[
 \gamma^2 =\sum_{\pi,\pi^\prime \in S_h} q_Y(\pi) \sigma_{\pi,\pi^\prime}^{11} q_Y(\pi^\prime)
+2\sum_{\pi,\pi^\prime \in S_h} q_X(\pi) \sigma_{\pi,\pi^\prime}^{12} q_Y(\pi^\prime)
+\sum_{\pi,\pi^\prime \in S_h} q_X(\pi) \sigma_{\pi,\pi^\prime}^{22} q_X(\pi^\prime)
\]
\end{theorem}

If we want to apply the above limit theorems for hypothesis testing and the determination of confidence intervals, we need to estimate the limit covariance matrix $\mathbf{\Sigma}$. We will again apply the kernel estimate, proposed by De Jong and Davidson (2000), using the same kernel $k$ and the same bandwidth 
$(b_n)_{n\geq 1}$ as before. We define the $\bbr^{2(h+1)!}$-valued random vectors
\[
 \mathbf{V}_i=\left( \left(1_{\{ \pi(X_i,\ldots,X_{i+h})=\pi\} } -\hat{q}_X(\pi)\right)_{\pi\in S_{h+1}},
\left(1_{\{ \pi(Y_i,\ldots,Y_{i+h})=\pi\} } -\hat{q}_Y(\pi)\right)_{\pi\in S_{h+1}}
 \right)^T.
\]
The kernel estimator for the covariance matrix $\mathbf{\Sigma}$ is then given by
\[
\hat{\mathbf{\Sigma}}_n=\frac{1}{n}\sum_{i=1}^{n-h}\sum_{j=1}^{n-h} k\left(\frac{i-j}{b_n} \right) \mathbf{V}_i\, \mathbf{V}_j^T.
\]
We denote the entries of estimated covariance matrix $\hat{\mathbf{\Sigma}}_n$ by $\hat{\sigma}^{11}_{\pi,\pi^\prime}(n)$,$\hat{\sigma}^{12}_{\pi,\pi^\prime}(n)$, and  $\hat{\sigma}^{22}_{\pi,\pi^\prime}(n)$, respectively, where $\pi,\pi^\prime \in S_{h+1}$. Plugging the estimated covariances into the formula for $\gamma^2$, we then obtain a consistent estimator for $\gamma^2$.
\[
 \hat{\gamma}_n^2 =\! \sum_{\pi,\pi^\prime \in S_h}\! \! \hat{q}_Y(\pi) \hat{\sigma}_{\pi,\pi^\prime}^{11}(n) \hat{q}_Y(\pi^\prime)
+2\! \sum_{\pi,\pi^\prime \in S_h}\!\! \hat{q}_X(\pi) \hat{\sigma}_{\pi,\pi^\prime}^{12}(n) \hat{q}_Y(\pi^\prime)
+\!\sum_{\pi,\pi^\prime \in S_h}\! \!\hat{q}_X(\pi) \hat{\sigma}_{\pi,\pi^\prime}^{22}(n) \hat{q}_X(\pi^\prime).
\]

\begin{corollary}
Under the same assumptions as in Theorem~\ref{th:q_hat}
\[
  \frac{\sqrt{n}(\hat{q}_n-q)}{\hat{\gamma}_n} \stackrel{\mathcal{D}}{\longrightarrow} N(0,1).
\]
As a consequence, $[\hat{q}_n-z_\alpha \hat{\gamma}_n,\hat{q}_n-z_\alpha \hat{\gamma}_n]$ is a confidence interval
with asymptotic coverage probability $(1-\alpha)$. Here $z_\alpha$ denotes the upper $\alpha$ quantile of the standard normal distribution.
\end{corollary}

\begin{remark}
(i) The coefficients $r,s$ can be estimated in the same way as $p,q$, by applying the estimators for $p$ and $q$ to the process $(X_i,-Y_i)_{i\geq 1}$. E.g., we can estimate $r$ by
\[
  \hat{r}_n=\frac{1}{n}\sum_{i=1}^{n-h} 1_{\{ \Pi(X_i,\ldots,X_{i+h})=\Pi(-Y_i,\ldots,-Y_{i+h} )  \} }
\]
(ii) The standardized ordinal pattern coefficient $\ord(X,Y)$ can be estimated by the plug-in estimator 
\[
  \hat\ord(X,Y)=\left(\frac{\hat{p}_n-\hat{q}_n}{1-\hat{p}_n}\right)^+
   -\left(\frac{\hat{r}_n-\hat{s}_n}{1-\hat{s}_n}\right)^+.
\]
If $p\neq q$ and $r\neq s$, we can establish asymptotic normality of $\hat\ord(X,Y)$ using the delta method. 
\end{remark}

\section{Weighted Ordinal Pattern Dependence}

As we have pointed out in the introduction, positive ordinal pattern dependence only counts the occurrence of coincident patterns. In the case of large values of $h$ it might happen that there is a strong co-movement of the two time series under consideration, which is distracted by a small noise. This might lead to `almost similar' patterns, a term which will be made precise below.

Besides, we might be interested in an entirely different dependence structure between two time series. We will see that our general approach yields a powerful tool to introduce and analyze various kinds of dependence. 

\subsection{Different Kinds of Dependence}

Let us begin with the following example: 

\begin{figure}[h]
\begin{minipage}[h]{0.47\textwidth}
\includegraphics[width=7cm, angle=0, trim= 0 4.5cm 0 4cm]{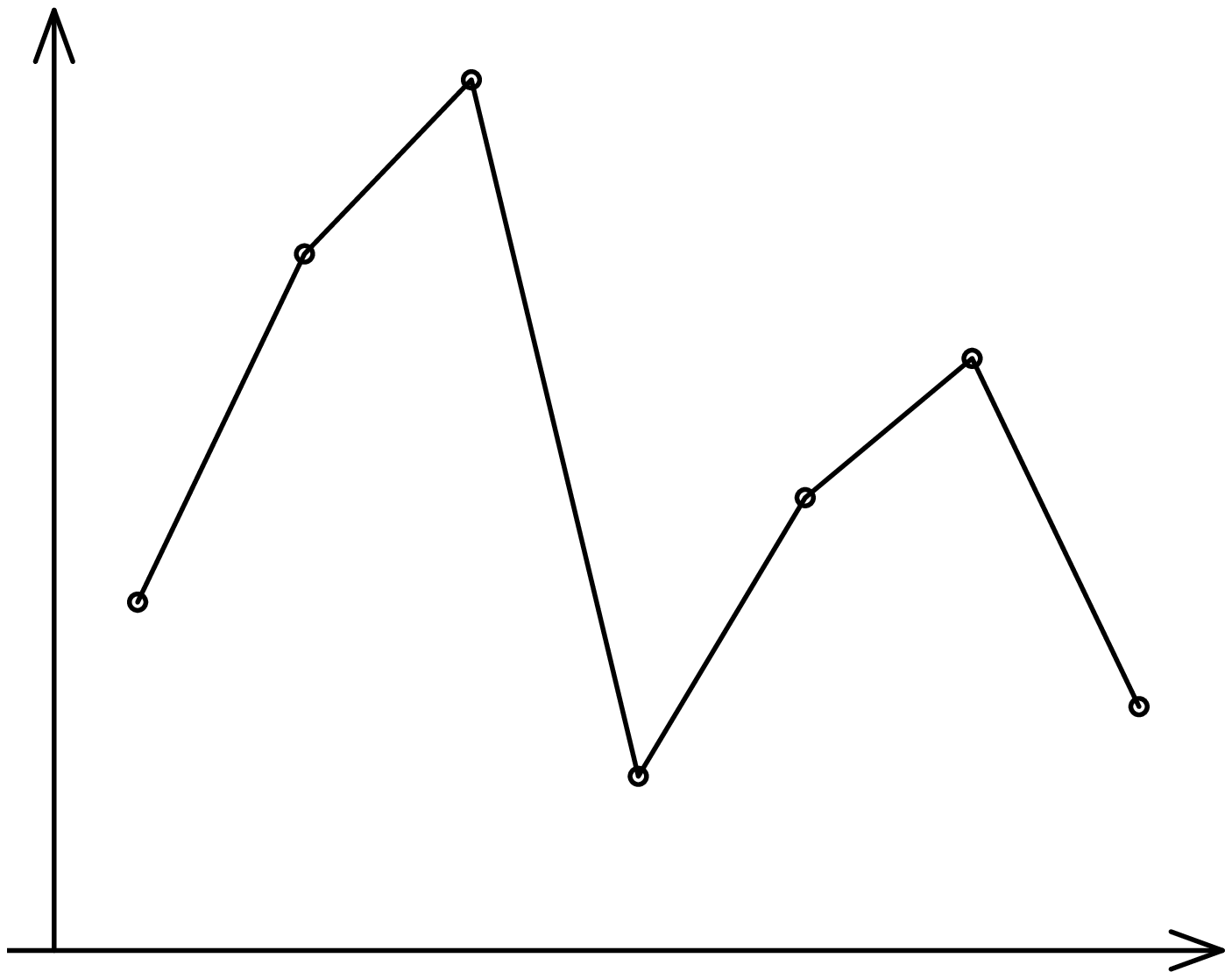}
\caption{The pattern $(2,1,5,4,0,6,3)$.}
\end{minipage}
\hfill
\begin{minipage}[h]{0.47\textwidth}
\includegraphics[width=7cm, angle=0, trim= 0 4.5cm 0 4cm]{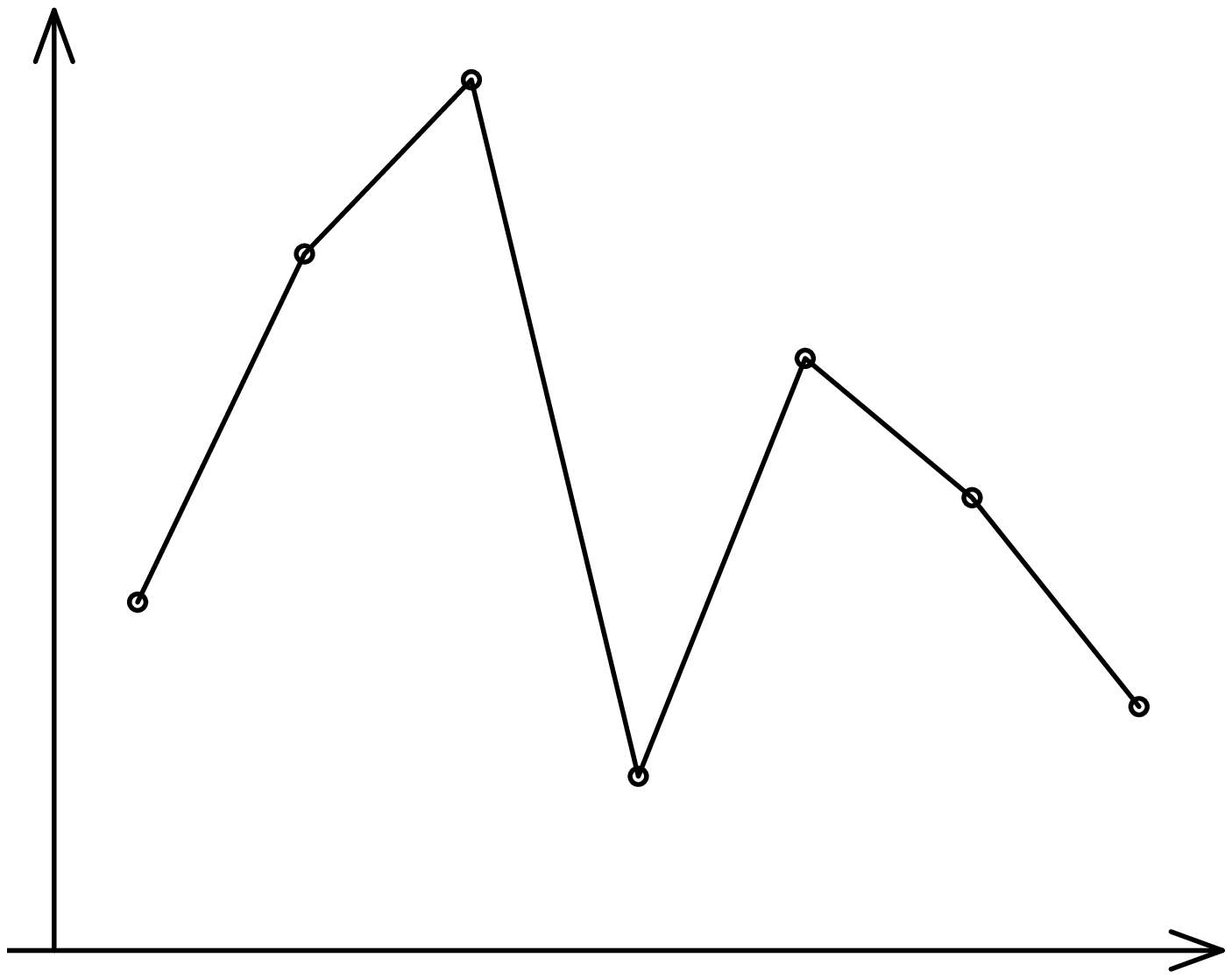}
\caption{The pattern $(2,1,4,5,0,6,3)$.}
\end{minipage}
\end{figure}

The difference between the corresponding permutations is only one neighboring transposition. In this sense the permutations are very close to each other. However, one has to be very careful with the meaning of `close to each other'. For $h=4$ the permutations $(1,3,2,0,4)$ and $(3,1,2,4,0)$ differ only by two transpositions. Nevertheless, they look almost like ord$\ominus$.

%On the space of permutations we define the following metric:
%\begin{definition}
%Let $(S_h,\circ)$ be the symmetric group. We define $d:S_h\times S_h\to [0,\infty[$ as follows: $d(\sigma, \tau)$ is the minimum number of pairwise adjacent transpositions to obtain $\tau$ from $\sigma$.
%\end{definition}
 
If the distance is chosen appropriately, (e.g. $\ell^1$-metric, see below)
a \emph{small} distance on $S_{h+1}$ can be interpreted as a \emph{strong} co-movement. Hence, we introduce a decreasing function $w$ on $d(S_{h+1},S_{h+1})$, the image of the metric. This function will be called \emph{weight function}. In the case described in Section 1 the distance function is the discrete metric and the weight function is $w=1_{\{0\}}$.

In the present section, we only consider positive dependence. The case of negative dependence can be treated in the same way, or by using positive dependence on $X$ and $-Y$.

Let us directly consider the most general setting. We will see below that it is sometimes useful to consider different types of patterns to be close to each other. Hence, we allow any metric on $S_{h+1}$, and even every pseudo-metric. The reduction of complexity is in general not as strong as in the `classical' setting as described in Section 1. The question is, how to compare the difference between the given data and the hypothetical case of independence. In the classical setting we have compared probabilities, but since it has been probabilities of Bernoulli random variables, we have compared expected values as well. This latter concept carries over to the more general case.

Let $d:S_{h+1}\times S_{h+1}\to \bbr_+$ be a pseudo-metric, that is, for every $\sigma,\pi,\rho\in S_{h+1}$ we have \\
\hspace*{2mm}(i) $d(\sigma,\sigma)=0$ \\
\hspace*{2mm}(ii) $d(\sigma, \pi)+d(\pi,\rho)\geq d(\sigma, \rho)$ \\
Furthermore, let $w: d(S_{h+1},S_{h+1})\to [0,1]$ be a monotonically decreasing function such that $w(0)=1$.
\begin{align} \begin{split}
E &w\Big(d\Big(\Pi(X_n,X_{n+1},...,X_{n+h}),\Pi(Y_n,Y_{n+1},...,Y_{n+h})\Big)\Big) \\
& -\sum_{\pi,\sigma\in S_{h+1}} w(d(\pi,\sigma))P\Big(\Pi(X_n,X_{n+1},...,X_{n+h})=\pi\Big) \cdot P\Big(\Pi(Y_n,Y_{n+1},...,Y_{n+h})=\sigma\Big)
\end{split}\end{align}
is called \emph{average weighted ordinal pattern dependence (AWOPD)}.  

The AWOPD can be normed as in \eqref{normed}. In applications it is sometimes more convenient to work without any norming: in the classical setting one compares the number of coincident patterns with the estimate of $q$ times the number of observed patterns, that is, the average number of coincident patterns one would expect under independence. Let $N$ be the number of observed points. In the new setting one compares the \emph{AWOPD-value}
\[
\sum_{j=1}^{N-h} w\Big(d\Big(\Pi(x_j,x_{j+1},...,x_{j+h}),\Pi(y_j,y_{j+1},...,y_{j+h})\Big)\Big)
\]
with the \emph{comparison value}, which is the (plug-in) estimate of
\[
\sum_{\pi,\sigma\in S_{h+1}} w(d(\pi,\sigma))P\Big(\Pi(X_n,X_{n+1},...,X_{n+h})=\pi\Big) \cdot P\Big(\Pi(Y_n,Y_{n+1},...,Y_{n+h})=\sigma\Big) 
\]
times $(N-h)$.

The first idea of this section has been to allow a small tolerance in comparing the ordinal structure of the two time series (cf. in particular Example \ref{ex:ellone} below). Another approach in this direction would be to use a kind of $\varepsilon$-band around the respective points of the time series. The problem here is in computation. We would lose the benefit which we get from analyzing only the ordinal structure, if we checked whether a point is in an $\varepsilon$-band around another point. And we would have to do this in fact for $h+1$ points simultaneously. 

There are various metrics, which can be defined on $S_{h+1}$ and which are used in different areas of mathematics. For a survey consult Deza and Huang (1998) and Critchlow (1985), Chapter II. The metric of choice depends highly on the application one has in mind. Here, we describe two different settings with suitable metrics.

\begin{example} \label{ex:ellone}
Maybe, the most natural choice for a metric is the $\ell^1$-distance: for $\pi,\sigma\in S_{h+1}$ define
\[
d_{\ell^1}(\sigma, \pi):= \sum_{j=0}^h \abs{\pi^{(j)}-\sigma^{(j)}}.
\]
Using this metric is in line with our interpretation from above. We are still interested (as in the first section) in the co-movement of time series, but we allow for a small tolerance. We do not have to have exactly the same pattern in both time series; it is enough if they are \emph{close} to each other. 

We emphasize the advantage of the new approach with a real world example:
let us consider again the relation between the S\&P 500 and the corresponding VIX (cf. Example \ref{ex:vix}).
Since we are dealing with a (generalized) ord$\ominus$, we analyze positive dependence between SPX and -VIX.

In Schnurr (2014) it was shown that there is a strong ord$\oplus$ between the two time series under consideration (up to the order $h=7$). Let us consider the time period from 06-12-1995 to 05-12-1997 ($n=500$) and fix the order $h=6$. For small orders the effect is not as strong. In the given data sets we find 15 coincident patterns. The classical comparison value is 0.7633. This means that if the two time series with their estimated pattern probabilities were independent, we would expect 0.7633 coincident patterns. In fact we find 15 of them. There is strong evidence for a classical positive ordinal pattern dependence. 

Next we use the $\ell^1$-metric and the weight function
\[
    w:= 1_{\{0\}}+0.75\cdot 1_{\{2\}}+0.50\cdot 1_{\{4\}}+0.25\cdot 1_{\{6\}}.
\]
Recall that the $\ell^1$-distance on $S_{h+1}$ is always an even number. We compare the AWOPD-value with the comparison value: the AWOPD-value is 101.5. Keep in mind that here not only the 15 coincident patterns are counted, but also various almost coincident patterns do count towards this score. The comparison value is in this case 13.5. The advantage of our new approach can be emphasized by analyzing a noisy version of the time series. We calculate the realized variance $V$ of the first time series (S\&P500) and add to it a white noise sequence of i.i.d. Gaussian random variables with mean zero and variance $V$, that is, the variance of the noise is as big as the variance of the data. We have simulated this 100 times. For 81 of the new noisy time series we found not a single coincident pattern between these time series and the original -VIX data. Using weighted ordinal pattern dependence with $d$ and $w$ as above we get a mean of 8.4125 for the AWOPD-value, which is significantly higher than the mean of the comparison value 3.983 (sd=0.4933). Even under a strong noise the AWOPD approach still detects the positive dependence. This example shows how the tolerance which we have included in the comparison makes our method more robust. 
\end{example}

While the metric in the example considered above is quite canonical, the (optimal) choice of the weight function is an interesting open question. With our linear function above we have been quite successful, although it was chosen ad hoc. 

On some occasions, it might not be co-movement we are interested in. In the following example, we use the metric in order to distinguish between `order' and `chaos'.

\begin{example}
Sometimes one might not be interested in co-monotonic or anti-monotonic behavior. Some time series, in particular in biology/medicine have times of regular behavior and others of, say, chaotic or turbulent behavior. It could be of interest to analyze whether the chaotic parts within two time series start and stop at about the same time. A structural break would then mean that one of the series starts its chaotic behavior while the other one is still in a regular regime or the other way around.   

First we have to answer the question how to measure `regular' behavior in terms of ordinal patterns. Secondly, we will introduce a (pseudo-)metric which describes how far away a pattern is from being regular. Finally we have to check whether regular parts and chaotic parts appear at the same time in given time series.

Let us start with the first question: it is doubtless that the pattern in Figure 10 shows a regular behavior (a time of growth) while the pattern in Figure 11 shows in a certain sense chaotic behavior, that is, it changes direction all the time. 
\begin{figure}[h]
\begin{minipage}[h]{0.47\textwidth}
\includegraphics[width=7cm, angle=0, trim= 0 4.5cm 0 4cm]{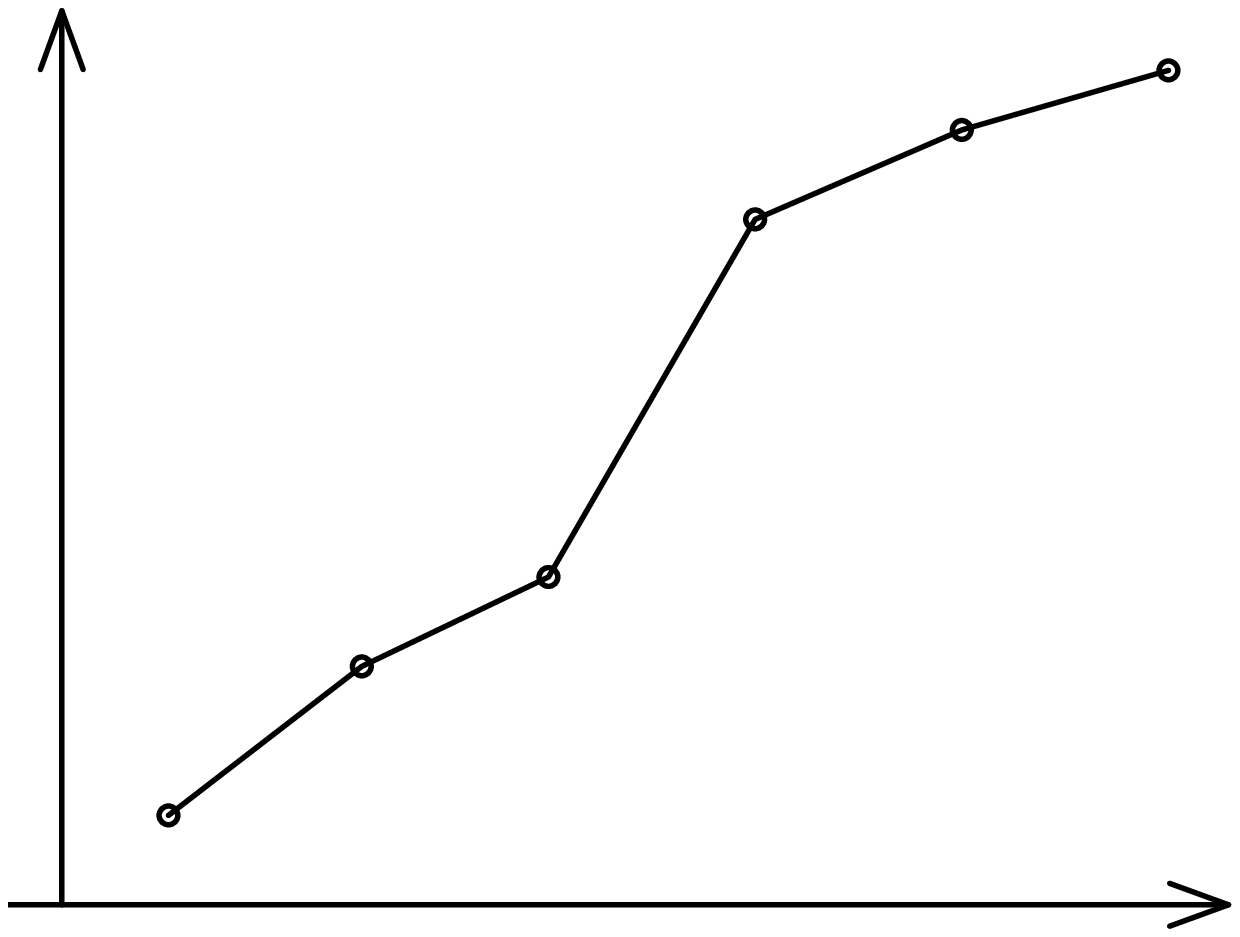}
\caption{The pattern $(5,4,3,2,1,0)$.}
\end{minipage}
\hfill
\begin{minipage}[h]{0.47\textwidth}
\includegraphics[width=7cm, angle=0, trim= 0 4.5cm 0 4cm]{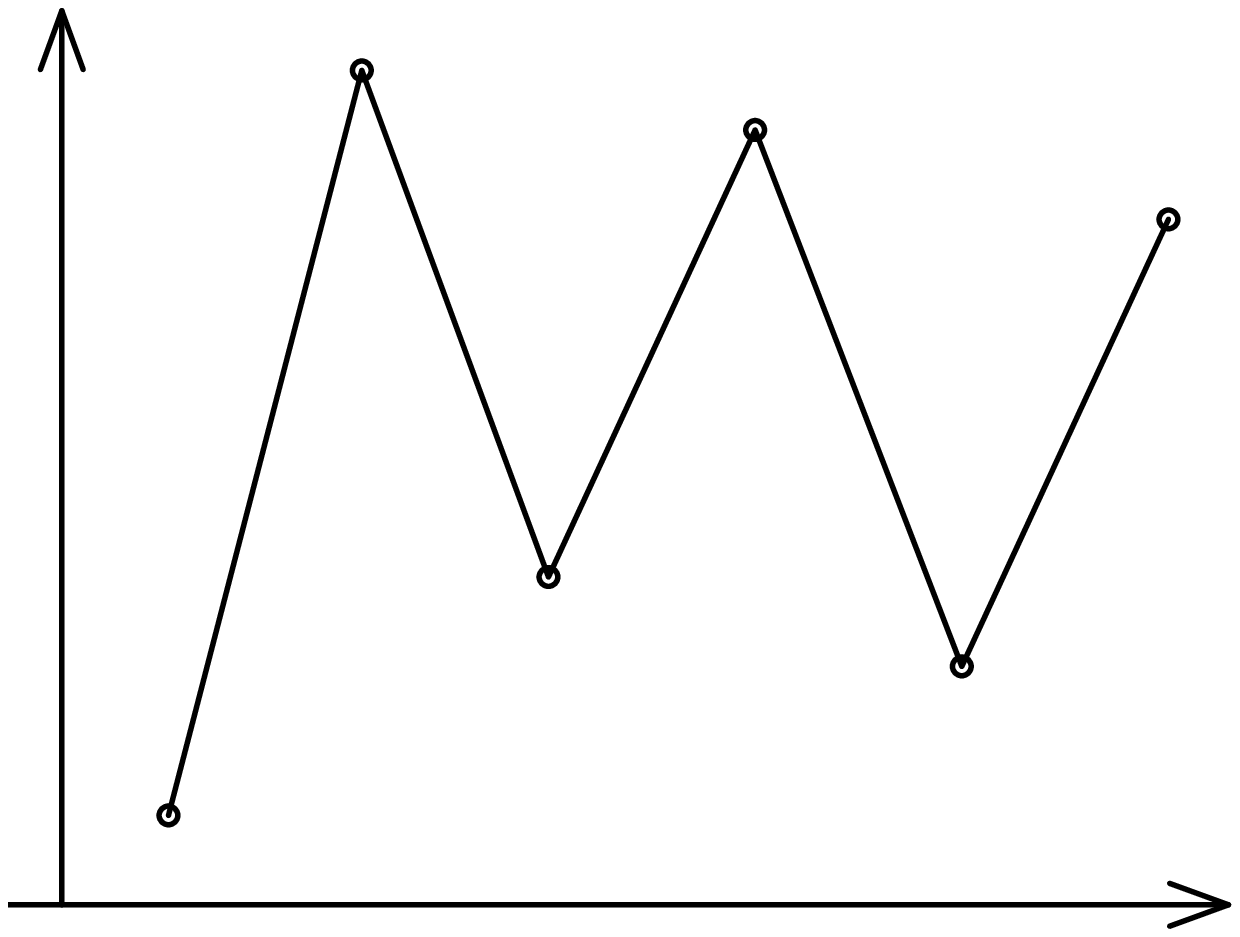}
\caption{The pattern $(1,3,5,2,4,0)$.}
\end{minipage}
\end{figure}
We want to make this mathematically tractable. The monotone patterns $(0,1,...,h)$ and $(h, h-1, ...,0)$ are from our point of view most regular. To every pattern in $S_{h+1}$ we assign the value 
\[
  \pi\mapsto c(\pi):=\min \big\{ d_{\ell^1}(\pi,(0,1,...,h)) ,  d_{\ell^1}(\pi,(h,h-1,...,0))\big\},
\]
that is, we measure how far away the pattern is from the two most regular ones. Again we use the $\ell^1$-distance, but in a different way than in our above example. 

A big value of $c(\pi)$ means that $\pi$ is `chaotic'. Chaotic patterns should be close to other chaotic patterns while regular patterns should be close to other regular ones. This is obtained in the following way: we use a metric on the space $c(S_{h+1})\subseteq \bbr_+$ which takes the order structure of $\bbr_+$ into account. W.l.o.g. we use the Euclidean distance $d_E$. Since we want to measure the distance between patterns, we pull this metric back via the function $c$:
\[
  d(\pi,\sigma):=d_E(c(\pi),c(\sigma))
\]
It is easy to check that $d$ is a pseudo-metric on $S_{h+1}$. Unlike in the case of a (proper) metric, $d(\sigma,\pi)=0$ does not imply $\sigma=\pi$. For example, the distance between the patterns $(0,1, ..., h)$ and $(h,h-1,...,0)$ is zero. The last step is identical to our previous example: we use a monotonically decreasing weight function in order to guaranty that a similar behavior results in a big AWOPD-value while a different behavior results in a small AWOPD-value. Using this method of analysis on real data is part of ongoing research. 
\end{example}

The main difference between the two examples is that different patterns are thought of as being close to each other. Our above approach hence gives us a lot of flexibility in terms of the kind of dependence which can be analyzed.  

\subsection{Limit Theorems and Structural Breaks in the Generalized Setting}

We estimate the AWOPD by the sample analogue
\[
  \hat{D}_n=\frac{1}{n} \sum_{i=1}^{n-h} w(d(\Pi(X_i,\ldots,X_{i+h}),\Pi(Y_i,\ldots,Y_{i+h})))
 -\sum_{\pi,\sigma \in S_{h+1}} w(d(\pi,\sigma))\hat{q}_X(\pi) \hat{q}_Y(\sigma),
\]
where $\hat{q}_X$ and $\hat{q}_Y$ are defined as in \eqref{eq:q_x_hat} and \eqref{eq:q_y_hat}, respectively. In order to derive the asymptotic distribution of $\hat{D}_n$, we note that $\hat{D}_n$ is a functional of the $(2\, (h+1)!+1)$-dimensional random vector 
\[
 \left(\frac{1}{n}\sum_{i=1}^{n-h}w(d(\Pi(X_i,\ldots,X_{i+h}),\Pi(Y_i,\ldots,Y_{i+h}))),
  (\hat{q}_X(\pi))_{\pi\in S_{h+1}}, (\hat{q}_Y(\pi))_{\pi\in S_{h+1}}  \right).
\]
We will now derive the asymptotic distribution of this random vector.
\begin{theorem}
Under the same assumptions as in Theorem~\ref{th:p_hat_ad}, 
\[
  \left(\frac{1}{n}\sum_{i=1}^{n-h}w(d(\Pi(X_i,\ldots,X_{i+h}),\Pi(Y_i,\ldots,Y_{i+h}))),
  (\hat{q}_X(\pi))_{\pi\in S_{h+1}}, (\hat{q}_Y(\pi))_{\pi\in S_{h+1}}\right)
\]
is asymptotically normal with mean vector
\[
  \left(E(w(d(\Pi(X_1,\ldots,X_{h+1}),\Pi(Y_1,\ldots,Y_{h+1})))), (q_X(\pi))_{\pi\in S_{h+1}},(q_Y(\pi))_{\pi\in S_{h+1}}   \right),
\]
and covariance matrix $({1}/{n})\cdot\mathbf{\Sigma}$, where
\begin{equation}
 \mathbf{\Sigma}=\left(
\begin{array}{ccc}
  a & \mathbf{b}_1^\prime & \mathbf{b}_2^\prime \\
   \mathbf{b}_1 &\mathbf{\Sigma}_{11} & \mathbf{\Sigma}_{12}\\
 \mathbf{b}_2 &\mathbf{\Sigma}_{21} &\mathbf{\Sigma}_{22}
 \end{array}
    \right).
\label{eq:awopd-sigma}
\end{equation}
Here $\mathbf{\Sigma}_{11}$, $\mathbf{\Sigma}_{12}$, and $\mathbf{\Sigma}_{22}$ are the $(h+1)!\times (h+1)!$ matrices with entries defined in \eqref{eq:sig_11}, \eqref{eq:sig_12}, and \eqref{eq:sig_22}, and $a\in \bbr$ and $\mathbf{b}_1,\mathbf{b}_2\in \bbr^{(h+1)!}$ are defined as follows
\begin{eqnarray*}
a&=& \sum_{i=-\infty}^\infty
   \Cov\big(w(d(\Pi(X_1,\ldots,X_{1+h}),\Pi(Y_1,\ldots,Y_{1+h}))), \\
   &&\qquad \qquad \qquad w(d(\Pi(X_i,\ldots,X_{i+h}),\Pi(Y_i,\ldots,Y_{i+h})))\big), \\
 \mathbf{b}_1(\pi) &=& \sum_{i=-\infty}^\infty \Cov\big(w(d(\Pi(X_1,\ldots,X_{1+h}),\Pi(Y_1,\ldots,Y_{1+h}))), 
  1_{\{\Pi(X_i,\ldots,X_{i+h})=\pi\}}\big)\\
\mathbf{b}_2(\pi) &=& \sum_{i=-\infty}^\infty \Cov\big(w(d(\Pi(X_1,\ldots,X_{1+h}),\Pi(Y_1,\ldots,Y_{1+h}))), 
  1_{\{\Pi(Y_i,\ldots,Y_{i+h})=\pi\}}\big)
\end{eqnarray*}
\label{th:awopd}
\end{theorem}

\begin{proof}
This follows from the multivariate  CLT for  functionals of mixing processes, which can be derived from Theorem~18.6.3 of Ibragimov and Linnik (1971) by using the Cram\'er-Wold device.
\end{proof}

 \begin{theorem}
Under the same assumptions as in Theorem~\ref{th:p_hat_ad}, 
\[
  \sqrt{n} (\hat{D}_n -AWOPD)  \rightarrow N(0,\gamma^2),
\] 
where the asymptotic variance $\gamma^2$ is given by the formula
\[
 \gamma^2 = \mathbf{\alpha}' \mathbf{\Sigma} \mathbf{\alpha},
\]
with $\mathbf{\Sigma}$ as in \eqref{eq:awopd-sigma} and 
\[
 \mathbf{\alpha}=\left(1,-\left(\sum_{\sigma \in S_{h+1}} w(d(\pi,\sigma))q_Y(\sigma)\right)_{\pi \in S_{h+1}},
-\left(\sum_{\sigma \in S_{h+1}} w(d(\sigma,\pi))q_X(\sigma)\right)_{\pi \in S_{h+1}}\right)'.
\]
\label{th:D_hat}
\end{theorem}
\begin{proof}
This follows from Theorem~\ref{th:awopd} and the delta method, applied to the function $g:\bbr^{2(h+1)!+1}\rightarrow \bbr$, given by
\[
 g(u,(v_\pi)_{\pi\in S_{h+1}},(w_\pi)_{\pi\in S_{h+1}})
 =u-\sum_{\pi,\sigma\in S_{h+1}} w(d(\pi,\sigma)) v_\pi w_\sigma.
\]
Note that $\alpha=\nabla g(E(w(d(\Pi(X_1,\ldots,X_{h+1}),\Pi(Y_1,\ldots,Y_{h+1})))), (q_X(\pi))_{\pi\in S_{h+1}},(q_Y(\pi))_{\pi\in S_{h+1}})$.
\end{proof}

Finally, we propose a test for structural breaks in the AWOPD rejecting for large values of the test statistic
\begin{eqnarray}
W_n&=&\max_{0\leq k\leq n-k} \frac{1}{\sqrt{n}}
 \sum_{i=1}^k \Big[ w\left(d(\Pi(X_i,\ldots,X_{i+h}),\Pi(Y_i,\ldots,Y_{i+h}))\right)\nonumber \\
 &&\qquad -\frac{1}{n} \sum_{i=1}^n
w\left(d(\Pi(X_i,\ldots,X_{i+h}),\Pi(Y_i,\ldots,Y_{i+h}))\right)
 \Big]
\label{eq:awopd-sbtest}
\end{eqnarray}

\begin{theorem}
Under the same assumptions as in Theorem~\ref{th:p_hat_ad}, and under the hypothesis of no structural break, 
\[
 W_n\stackrel{\mathcal{D}}{\longrightarrow} \sqrt{a} \sup_{0\leq \lambda\leq 1}
|W(\lambda)-\lambda W(1)|,
\]
where $a$ is defined as in Theorem~\ref{th:awopd}. 
\end{theorem}
\begin{proof}
The proof follows along the lines of the proof of Theorem~\ref{th:cp_test_ad}. We introduce the process
\begin{eqnarray*}
 W_n(\lambda)&=&\frac{1}{\sqrt{n}}
 \sum_{i=1}^{[n\lambda]} \Big[ w\left(d(\Pi(X_i,\ldots,X_{i+h}),\Pi(Y_i,\ldots,Y_{i+h}))\right) \\
 &&\qquad -\frac{1}{n}\sum_{i=1}^n
w\left(d(\Pi(X_i,\ldots,X_{i+h}),\Pi(Y_i,\ldots,Y_{i+h}))\right)
 \Big]\\
&=&\frac{1}{\sqrt{n}} \sum_{i=1}^{[n\lambda]}
w\left(d(\Pi(X_i,\ldots,X_{i+h}),\Pi(Y_i,\ldots,Y_{i+h}))\right) \\
&&\qquad 
-\frac{[n\lambda]}{n}\frac{1}{\sqrt{n}} \sum_{i=1}^n w\left(d(\Pi(X_i,\ldots,X_{i+h}),\Pi(Y_i,\ldots,Y_{i+h}))\right).
\end{eqnarray*}
As in the proof of Theorem~\ref{th:cp_test_ad}, we can show that $(W_n(\lambda))_{0\leq \lambda \leq 1}$ converges in distribution to the process $\sqrt{a}(W(\lambda)-\lambda W(1))_{0\leq \lambda \leq 1}$.
\end{proof}
\section{Proofs}

Let us first show that the time series in Example \ref{ex:oneapprox} are 1-approximating functionals: In part (i) we have considered a $MA(\infty)$-time series $X$. We define the functions
$f_m:\bbr^{2m+1}\rightarrow \bbr$ by
$ f_m (z_{i-m},\ldots,z_{i+m})=\sum_{j=0}^m \alpha_j z_{i-j}$.
Thus we obtain
\begin{eqnarray*}
 E|X_i-f_m(Z_{i-m},\ldots,Z_{i+m})| &=& E\left|\sum_{j=m+1}^\infty \alpha_j Z_{i-j}\right| \\
&\leq& \left(  E\Big(\sum_{j=m+1}^\infty \alpha_j Z_{i-j}^2  \Big) \right)^{1/2} 
= \left(\sum_{j=m+1}^\infty a_j^2  \right)^{1/2} \sqrt{\Var(Z_1)}.
\end{eqnarray*}
Hence, $(X_i)_{i\geq 1}$ is a $1$-approximating functional with coefficients 
$a_m=\left(\sum_{j=m+1}^\infty \alpha_j^2  \right)^{1/2} $. 

In part (ii) we have considered the baker's map. We will now show that $(X_n)_{n\geq 0}$ is a functional of an i.i.d. process. It is well-known that for almost every $\omega\in [0,1]$, there is a unique dyadic expansion 
\[
 \omega=\sum_{j=1}^\infty \frac{Z_j}{2^j},
\]
where $Z_j=Z_j(\omega)\in \{0,1\}$. Moreover, the random variables $(Z_j)_{j\geq 1}$ are i.i.d. and $P(Z_i=0)=P(Z_i=1)=1/2$. Note that 
\[
  T^n(\omega)=\sum_{j=1}^\infty \frac{Z_{j+n}}{2^j}=\sum_{j=-\infty}^{-1} 2^j Z_{n-j}.
\]
We then define $f_m(z_{-m},\ldots,z_m)=g(\sum_{j=-m}^{-1}2^j z_j  )$, and thus we obtain
\begin{eqnarray*}
  E|X_0 -f_m(Z_{-m},\ldots,Z_m)|
 &=& E\left|g\left( \sum_{j=1}^\infty Z_{j}{2^j}\right) -g\left(\sum_{j=1}^{m}\frac{Z_j}{2^j} \right) \right| \\
 &\leq & \| g\|_L E\left(\sum_{j=m+1}^\infty \frac{Z_j}{2^j}\right)=\|g\|_L \frac{1}{2^{m+1}}.
\end{eqnarray*}
Hence, $(X_n)_{n\geq 0}$ is a $1$-approximating functional of the i.i.d. process $(Z_j)_{j\in \bbz}$ with approximating constants $a_m={\|g\|_L}/{2^{m+1}}$.
The proof for the claim of part (iii) is similar and hence omitted. 

The following lemma is used in the proof of Theorem \ref{th:p_hat_ad}.
\begin{lemma}\label{lem:oneapprox}
Let $(X_i,Y_i)_{i\geq 1}$ be a $1$-approximating functional of the  process $(Z_i)_{i\in \bbz}$ with approximating coefficients $(a_m)_{m\geq 1}$.\\[1mm]
(i) Assume that the distribution functions of $X_i-X_1$ is Lipschitz-continuous, for any $i\in \{1,\ldots,h+1\}$. Then, for any permutation $\pi$,
\[
 1_{\{\Pi(X_i,X_{i+1},\ldots X_{i+h})= \pi \}}
\]
is a $1$-approximating functional with approximation constants $(O(\sqrt{a_m}))_{m\geq 1}$.
\\[1mm]
(ii)
Assume that the distribution functions of $X_i-X_1$, and of $Y_i-Y_1$, are Lipschitz-continuous, for any $i\in \{1,\ldots,h+1\}$. Then, 
\[
 1_{\{\Pi(X_i,X_{i+1},\ldots X_{i+h})= \Pi(Y_i,Y_{i+1},\ldots Y_{i+h})\}}
\]
is a $1$-approximating functional with approximation constants $(O(\sqrt{a_m}))_{m\geq 1}$.
\label{le:1-approx}
\end{lemma}

\begin{proof} We only present the proof of part (ii). The proof of part (i) follows the same lines.
Let $m\geq 1$ and define $(X_i^{(m)},Y_i^{(m)})=f_m(Z_{i-m},\ldots,Z_{i+m})$. Then, 
the following inequality holds
\begin{eqnarray*}
 &&1_{\{\Pi(X_i,X_{i+1},\ldots X_{i+h})= \Pi(Y_i,Y_{i+1},\ldots Y_{i+h})\}}
-1_{\{\Pi(X_i^{(m)},X_{i+1}^{(m)},\ldots X_{i+h}^{(m)})= \Pi(Y_i^{(m)},Y_{i+1}^{(m)},\ldots Y_{i+h}^{(m)})\}}\\
&& \quad \leq \sum_{j=0}^h 1_{\{|X_{i+j}-X_{i+j}^{(m)}|>\epsilon\}} 
+ \sum_{j=0}^h 1_{\{|Y_{i+j}-Y_{i+j}^{(m)}|>\epsilon\}} \\
&& \qquad+ \sum_{0\leq j \neq k\leq h} 1_{\{|X_{i+j}-X_{i+k}|\leq 2\epsilon \}}
+ \sum_{0\leq j \neq k\leq h} 1_{\{|Y_{i+j}-Y_{i+k}|\leq 2\epsilon \}}.
\end{eqnarray*}
Thus, by stationarity, we obtain 
\begin{eqnarray*}
&& E \left| 1_{\{\Pi(X_i,X_{i+1},\ldots X_{i+h})= \Pi(Y_i,Y_{i+1},\ldots Y_{i+h})\}}
-1_{\{\Pi(X_i^{(m)},X_{i+1}^{(m)},\ldots X_{i+h}^{(m)})= \Pi(Y_i^{(m)},Y_{i+1}^{(m)},\ldots Y_{i+h}^{(m)})\}}
\right|\\
&& \quad \leq (h+1) P(|X_1-X_1^{(m)}|\geq \epsilon) +(h+1) P(|Y_1-Y_1^{(m)}| \geq \epsilon)\\
&& \qquad +\sum_{1\leq j\neq k\leq h} P(|X_j-X_k| \leq 2\epsilon)
+\sum_{1\leq j\neq k\leq h} P(|Y_j-Y_k| \leq 2\epsilon)\\
&&\quad \leq 2(h+1) \frac{a_m}{\epsilon}+2(h+1)h C\epsilon.
\end{eqnarray*}
Choosing $\epsilon=\sqrt{a_m}$, we thus obtain
\[
 E \left| 1_{\{\Pi(X_i,X_{i+1},\ldots X_{i+h})= \Pi(Y_i,Y_{i+1},\ldots Y_{i+h})\}}
-1_{\{\Pi(X_i^{(m)},X_{i+1}^{(m)},\ldots X_{i+h}^{(m)})= \Pi(Y_i^{(m)},Y_{i+1}^{(m)},\ldots Y_{i+h}^{(m)})\}}
\right| \leq C\sqrt{a_m},
\]
thus showing that $1_{\{\Pi(X_i,X_{i+1},\ldots X_{i+h})= \Pi(Y_i,Y_{i+1},\ldots Y_{i+h})\}}$ is a $1$-approximating functional.
\end{proof}

\begin{proof}[Proof of Theorem \ref{th:q_hat}]
We define the function $f:\bbr^{2h!}\rightarrow \bbr$ by $ f(\mathbf{x},\mathbf{y})=\sum_{i=1}^{h!} x_i\, y_i $,
where $\mathbf{x}=(x_1,\ldots,x_{h!})$, $\mathbf{y}=(y_1,\ldots,y_{h!})$. $f$ is everywhere differentiable, with partial derivatives
$\frac{\partial}{\partial x_i} f(x,y)=y_i$ and $\frac{\partial}{\partial y_i} f(x,y)=x_i$. Thus, denoting by
$\nabla_x f$ the vector of partial derivatives of $f$ with respect to the $x$-coordinates,  
we obtain 
\[
 \nabla f(x,y)=\left(
 \begin{array}{c}
  \nabla_x f \\
  \nabla_y f
  \end{array}
 \right) (\mathbf{x},\mathbf{y}) 
 =\left(
\begin{array}{c}
   \mathbf{y} \\
   \mathbf{x}
  \end{array}
 \right). 
\]
Observe that $\hat{q}_n = f((\hat{q}_X(\pi))_{\pi \in S_h}, (\hat{q}_Y(\pi))_{\pi \in S_h})$, and that
$q=f(q_X,q_Y)$. Hence, we may apply the delta method, which yields $\sqrt{n} (\hat{p}_n - p)\rightarrow N(0,\gamma^2)$, where
\begin{eqnarray*}
 \gamma^2&=&(\nabla f(q_X,q_Y))^T \mathbf{\Sigma} \nabla f(q_X,q_Y) \\
&=&  (q_Y^T, q_X^T) 
\left(  
\begin{array}{cc}
\mathbf{\Sigma}_{11} & \mathbf{\Sigma}_{12} \\
\mathbf{\Sigma}_{12} & \mathbf{\Sigma}_{22}
\end{array}
\right)
\left(
\begin{array}{c}
   q_Y \\
   q_X
  \end{array}
 \right) \\
&=& q_Y^T \mathbf{\Sigma}_{11} q_Y +2 q_X^T \mathbf{\Sigma}_{12} q_Y + q_X^T \mathbf{\Sigma}_{22} q_X.
\end{eqnarray*}
Using \eqref{eq:sig_11}, \eqref{eq:sig_12}, and \eqref{eq:sig_22}, we then obtain the final formula for $\gamma^2$.
\end{proof}

\begin{proof}[Proof of Theorem \ref{th:cp_test_ad}]
We introduce the process 
\[
 T_n(\lambda)=\frac{1}{\sqrt{n}} \sum_{i=1}^{[n\lambda]} \left(1_{\{\Pi(X_i,\ldots,X_{i+h})=
\Pi(Y_i,\ldots,Y_{i+h})  \}} -\hat{p}_n\right),
\] 
which we can rewrite as follows
\begin{eqnarray*}
 T_n(\lambda)
&=& \frac{1}{\sqrt{n}} \sum_{i=1}^{[n\lambda]} \left(1_{\{\Pi(X_i,\ldots,X_{i+h})=
\Pi(Y_i,\ldots,Y_{i+h})  \}} -p\right) -\frac{[n\lambda]}{\sqrt{n}} \left(\hat{p}_n-p)\right)\\
&=& \frac{1}{\sqrt{n}} \sum_{i=1}^{[n\lambda]} \left(1_{\{\Pi(X_i,\ldots,X_{i+h})=
\Pi(Y_i,\ldots,Y_{i+h})  \}} -p\right) \\
&& \qquad -\frac{[n\, \lambda]}{n} \frac{1}{\sqrt{n}}\sum_{i=1}^{n} \left(1_{\{\Pi(X_i,\ldots,X_{i+h})=
\Pi(Y_i,\ldots,Y_{i+h})  \}} -p\right)\\
&&\qquad -\frac{[n\, \lambda]}{n^{3/2}}\sum_{i=n-h+1}^n1_{\{\Pi(X_i,\ldots,X_{i+h})=
\Pi(Y_i,\ldots,Y_{i+h})  \}} .
\end{eqnarray*}
Note that $({[n\, \lambda]}/{n^{3/2}})\sum_{i=n-h+1}^n1_{\{\Pi(X_i,\ldots,X_{i+h})= \Pi(Y_i,\ldots,Y_{i+h})} $ converges to zero in probability, and that ${[n\lambda]}/{n}$ converges to $\lambda$. Thus,
by the invariance principle for the partial sums of the indicator variables $1_{\{\Pi(X_i,\ldots,X_{i+h})=
\Pi(Y_i,\ldots,Y_{i+h})  \}}$,  we obtain immediately that the term $({1}/{\sqrt{n}})\sum_{i=1}^{[n\lambda]} (1_{\{\Pi(X_i,\ldots,X_{i+h})=\Pi(Y_i,\ldots,Y_{i+h})  \}}-p)$ converges in distribution to a Brownian motion with variance $\sigma^2$. Thus, by the continuous mapping theorem, we obtain convergence of $(T_n(\lambda))_{0\leq \lambda \leq 1}$ towards $\sigma(W(\lambda)-\lambda W(1))$. 
Another application of the continuous mapping theorem yields that 
\[
 \sup_{0\leq \lambda \leq 1} |T_n(\lambda)| \rightarrow \sigma \sup_{0\leq \lambda \leq 1}
  |W(\lambda)-\lambda W(1)|.
\]
Finally, we observe that $T_n=\sup_{0\leq \lambda \leq 1-\frac{h}{n}} |T_n(\lambda)|$, and that 
\begin{eqnarray*}
 \left|\sup_{0\leq \lambda \leq 1}|T_n(\lambda)| -\sup_{0\leq \lambda \leq 1-\frac{h}{n}}|T_n(\lambda)|\right| 
&\leq & \frac{1}{\sqrt{n}} \sum_{i=n-h+1}^n \left| 1_{\{\Pi(X_i,\ldots,X_{i+h})=\Pi(Y_i,\ldots,Y_{i+h})\}} -\hat{p}_n  \right|.
\end{eqnarray*}
As the right hand side converges to zero in probability, we have finally proved that $T_n$ converges in distribution to $\sigma\sup_{0\leq \lambda \leq 1}|W(\lambda)-\lambda W(1)|$.
\end{proof}

\section{Conclusion}

In the present paper we have introduced a new method to detect structural breaks in the dependence between two time series. To this end we have used the concept of ordinal pattern dependence which has been introduced in Schnurr (2014). While that article contained mainly a case study, here, we have presented the technical framework and generalized the concept substantially by using distance functions on the space of ordinal patterns. This allows us to analyze various kinds of dependence in future research.  

Our approach has several advantages compared to other ways of analyzing structural breaks within the dependence: the method is robust against measurement errors or small perturbations of the data. The intuition behind the concept is clear and there are quick algorithms in order to carry out the analysis. Let us emphasize that we do not need our random variables $X_i$ to have second moments which is a standard assumption for all tests which are based on correlation. 

It is important to note that even the classical ordinal pattern dependence does not measure the same phenomena as correlation measures. It is not in the scope of the present paper, but let us mention that we have analyzed data from medicine as well as hydrology which admit an ord$\oplus$ without showing a significant positive correlation and those with a strong positive correlation without showing ordinal pattern dependence. This statement remains true, comparing the ord$\oplus$ with kendall's tau or spearman's rho. 

Dealing with financial data we have seen that our test works in practice.  One could have other applications in mind. Since the method is scale free one can compare data coming from entirely different sources. As an example one could analyze the dependence between asset data and the heart rate of a trader. 

\bibliographystyle{Chicago}

\begin{thebibliography}{99}

\bibitem{bandt05}
Bandt, C. (2005):
\newblock Ordinal time series analysis.
\newblock {\em Ecological Modelling}, \textbf{182}, 229--238.

\bibitem{ban-pom02}
Bandt, C. and B. Pompe (2002):
\newblock Permutation entropy: A natural complexity measure for time series.
\newblock {\em Phys. Rev. Lett.}, \textbf{88} 174102 (4 pages).

\bibitem{ban-shi07}
Bandt, C. and F. Shiha (2007):
\newblock Order Patterns in Time Series.
\newblock {\em J. Time Ser. Anal.}, \textbf{28},  646--665.

\bibitem{bor-bur-deh2001}
Borovkova, S., R. Burton and H. Dehling (2001):
\newblock Limit theorems for functionals of mixing processes with applications to $U$-statistics and dimension estimation.
\newblock {\em Trans. Amer. Math. Soc.} \textbf{353}, 4261--4318. 

\bibitem{bro-dav91}
Brockwell, J. and R.A. Davis (1991):
\newblock {\em Time Series: Theory and Methods}.
\newblock Springer, New York.

\bibitem{critchlow}
Critchlow, D.E. (1985):
\newblock  {\em Metric methods for analyzing partially ordered data}.
\newblock Springer, Berlin.

\bibitem{dehlingetal2013}
Dehling, H., R. Fried, O. Sh. Sharipov, D. Vogel and M. Wornowizki (2013):
\newblock Estimation of the variance of partial sums of dependent processes. 
\newblock {\em Stat. Prob. Lett.} \textbf{83}, 141--147. 

\bibitem{dej-dav-00}
De Jong, R.M. and J. Davidson (2000):
\newblock Consistency of kernel estimators of heteroscedastic and autocorrelated covariance matrices.
\newblock {\em Econometrica} \textbf{68}, 407--423.

\bibitem{dez-hua}
Deza, M. and T. Huang (1998):
\newblock Metrics on Permutations, a Survey.
\newblock {\em Journal of Combinatorics, Information and System Sciences}, 14 pages.

\bibitem{dia-gra77}
Diaconis, P. and R.L. Graham (1977):
\newblock Spearman's footrule as a measure of disarray.
\newblock {\em J. Royal Stat. Soc., Ser. B} \textbf{39}, 262--268.

\bibitem{ibr-lin-71}
Ibragimov, I.A. and Yu.V. Linnik (1971):
\newblock {\em Independent and stationary sequences of random variables.}
\newblock Wolters-Noordhoff, Groningen.

\bibitem{kac46}
Kac, M. (1946):
\newblock{On the distribution of the values of sums of the type $\sum f(2^t)$}
\newblock{\em Annals of Mathematics}, \textbf{47},  33--49.

\bibitem{kel-sin-emo07}
Keller, K., M. Sinn and J. Emonds (2007):
\newblock Time Series from the Ordinal Viewpoint.
\newblock {\em Stochastics and Dynamics}, \textbf{2}, 247--272.

\bibitem{kel-sin05}
Keller, K. and M. Sinn (2005):
\newblock Ordinal Analysis of Time Series.
\newblock {\em Physica A}, \textbf{356}, 114--120.

\bibitem{kel-sin10}
Keller, K. and M. Sinn (2011):
\newblock Estimation of ordinal pattern probabilities in Gaussian processes with stationary increments.
\newblock {\em Comp. Stat. Data Anal.}, \textbf{55}, 1781--1790.

\bibitem{ordinalpatdep}
Schnurr, A. (2014):
\newblock An Ordinal Pattern Approach to Detect and to Model Leverage Effects and Dependence Structures Between Financial Time Series.
\newblock {\em Stat. Papers} \textbf{55(4)} (2014), 919--931. 

\bibitem{sin-gho-kel}
Sinn, M., A. Ghodsi,  and K. Keller (2012):
\newblock Detecting Change-Points in Time Series by Maximum Mean Discrepancy of Ordinal Pattern Distributions.
\newblock In: Proceedings of the 28th Conference on Uncertainty in Artificial Intelligence (UAI), 786--794. 

\bibitem{sin-kel-che}
Sinn, M., K. Keller, and B. Chen (2013):
\newblock Segmentation and classification of time series using ordinal pattern distributions.
\newblock {\em Eur. Phys. J. Special Topics} \textbf{222}, 587--598.

\bibitem{mad-yor11}
Madan, D.B. and M. Yor (2011):
\newblock The S\&P 500 Index as a Sato Process Travelling at the Speed of the VIX.
\newblock {\em Applied Mathematical Finance}, \textbf{18(3)}, 227--244. 

\bibitem{whaley}
Whaley, R.E. (2008):
\newblock Understanding VIX. 
\newblock Available at SSRN: http://ssrn.com/abstract=1296743.

\end{thebibliography}

\end{document}